\documentclass[a4paper, 10pt]{amsart}
\usepackage{amsmath}
\usepackage{amssymb, color}
\usepackage{url}
\usepackage{hyperref}

\theoremstyle{plain}
\newtheorem{theorem}{\bf Theorem}[section]
\newtheorem{proposition}[theorem]{\bf Proposition}
\newtheorem{lemma}[theorem]{\bf Lemma}

\theoremstyle{definition}

\newtheorem{examples}[theorem]{\bf Examples}

\newcommand{\N}{\mathbb N}
\newcommand{\Z}{\mathbb Z}
\newcommand{\R}{\mathbb R}

 \DeclareMathOperator{\ord}{ord}
 
 \DeclareMathOperator{\id}{id}

\DeclareMathOperator{\Ker}{Ker} 

\renewcommand{\time}{\negthinspace \times \negthinspace}

\newcommand{\DP}{\negthinspace : \negthinspace}

\newcommand{\red}{\text{\rm red}}

\newcommand{\norm}[1]{\ \parallel \negthinspace #1 \negthinspace \parallel \ }

\renewcommand{\t}{\, | \,}

\numberwithin{equation}{section}

\begin{document}

\title{The set of  distances in seminormal weakly Krull monoids}

\address{University of Graz, NAWI Graz, \\
Institute for Mathematics and Scientific Computing \\
Heinrichstra{\ss}e 36\\
8010 Graz, Austria}

\email{alfred.geroldinger@uni-graz.at, qinghai.zhong@uni-graz.at}

\author{Alfred Geroldinger  and Qinghai Zhong}

\thanks{This work was supported by
the Austrian Science Fund FWF, Project Numbers M1641-N26 and P26036-N26.}



\begin{abstract}
The set of distances  of a monoid or of a domain  is the set of all $d \in \N$ with the following property: there are irreducible elements $u_1, \ldots, u_k, v_1, \ldots, v_{k+d}$ such that $u_1 \cdot \ldots \cdot u_k = v_1 \cdot \ldots \cdot v_{k+d}$, but $u_1 \cdot \ldots \cdot u_k$ cannot be written as a product of $l$ irreducible
elements for any $l$ with $k < l < k+d$. We show that the set of distances is an interval for certain seminormal weakly Krull monoids which include seminormal orders in holomorphy rings of global fields.
\end{abstract}

\maketitle

\bigskip
\section{Introduction and Main Result}
\bigskip

Let $H$ be a $v$-noetherian monoid (for example, a noetherian domain). Then every non-unit of $H$ has a factorization as a finite product of atoms (irreducible elements), and all these factorizations are unique (i.e., $H$ is factorial) if and only if $H$ is a Krull monoid with trivial $v$-class group. Otherwise, there are elements having factorizations which differ not only up to associates and up to the order of the factors. The occurring phenomena of  non-uniqueness are described by arithmetical invariants such as sets of lengths and sets of distances. We  recall some arithmetical concepts and then we formulate the main result of the present paper.

For a finite non-empty set $L = \{m_1, \ldots, m_k\}$ of positive integers with $m_1 < \ldots < m_k$, we denote by $\Delta (L) = \{m_i-m_{i-1} \mid i \in [2,k] \}$ the set of distances of $L$. Thus $\Delta (L)=\emptyset$ if and only if $|L|\le 1$. If a non-unit $a \in H$ has a factorization $a = u_1 \cdot \ldots \cdot u_k$ into atoms $u_1, \ldots, u_k$, then $k$ is called the length of the factorization, and the set $\mathsf L (a)$ of all possible factorization lengths $k$ is called the set of lengths of $a$. In $v$-noetherian monoids all sets of lengths are finite. If there is an element $a \in H$ with $|\mathsf L (a)|>1$, then the $n$-fold sumset $\mathsf L (a) + \ldots + \mathsf L (a)$ is contained in $\mathsf L (a^n)$ whence  $|\mathsf L (a^n)| > n$ for every $n \in \N$.  The set of distances $\Delta (H)$ (also called the delta set of $H$) is the union of all sets $\Delta (\mathsf L (a))$ over all non-units $a \in H$. Thus, by definition, $\Delta (H)=\emptyset$ if and only if $|\mathsf L (a)|=1$ for all non-units $a\in H$, and $\Delta (H)=\{d\}$ if and only if $\mathsf L (a)$ is an arithmetical progression with difference $d$ for all non-units $a \in H$.

The set of distances (together with associated invariants, such as the catenary degree) has found wide interest in the literature in settings ranging from numerical monoids to Mori domains (for a sample out of many see \cite{Ch-Sc-Sm08b,C-G-L09,B-C-R-S-S10,Ge-Gr-Sc-Sc10, Ge-Gr-Sc11a, Ch-Go-Pe14a, C-C-M-M-P14, Co-Ka15a, Ge-Zh16a, Pl-Sc16a}). In the present paper we focus on seminormal weakly Krull monoids and show  -- under mild natural assumptions -- that their sets of distances are intervals.

\medskip
\begin{theorem} \label{1.1}
Let $H$ be a seminormal $v$-noetherian weakly Krull monoid, $\widehat H$ its complete integral closure, and $\emptyset \ne \mathfrak f = (H \DP \widehat H)$ its conductor. Suppose that the localization $H_{\mathfrak p}$ is finitely primary  for each minimal prime ideal $\mathfrak{p}\in\mathfrak{X}(H)$, and  that every
class of the $v$-class group $G=\mathcal C_v (H)$  contains a  minimal prime ideal $\mathfrak p \in \mathfrak X (H)$ with $\mathfrak p \not\supset \mathfrak f$. Then $\Delta (H)=\emptyset$ or $\min \Delta (H)=1$. Moreover, we have
\begin{enumerate}
\item If $|G|=1$, then $\Delta (H) \subset \{1\}$ and if $G$ is infinite, then $\Delta (H) = \N$.

\item Suppose that $G$ is finite. If there is at most one $\mathfrak p \in \mathfrak X (H)$ such that $|\{ \mathfrak P \in \mathfrak X (\widehat H) \mid \mathfrak P \cap H = \mathfrak p\}|>1$ or if $G$ is an elementary $2$-group, then $\Delta (H)$ is a finite interval.
\end{enumerate}
\end{theorem}

Seminormal orders in algebraic number fields satisfy all  assumptions of Theorem \ref{1.1}, and more examples will be given in Section \ref{2}. Seminormal orders  have been studied by Dobbs and Fontana in \cite{Do-Fo87}, where they provide, among others,  a full characterization of seminormal orders in quadratic number fields.
Note that  for a non-principal order $R$, which is not  seminormal and   whose Picard group has at most two elements, it is open whether or not we have $1 \in \Delta (R)$, let alone whether or not $\Delta (R)$ is an interval (e.g., \cite{Ph12b}).

Every Krull monoid is seminormal $v$-noetherian weakly Krull and all localizations are discrete valuation monoids and hence finitely primary. For Krull monoids having a minimal prime ideal in each class (whence in particular for principal orders in algebraic number fields) it is well-known that the set of distances is an interval (\cite{Ge-Yu12b}), and it is the goal of the present paper to generalize this result from the Krull to the weakly Krull case. Note, that  even in the case of Krull monoids, the assumption that every class contains a minimal prime ideal is essential to obtain that the set of distances is an interval (see Examples \ref{3.4}.1).

Suppose that $H$ is seminormal $v$-noetherian weakly Krull monoid with finite $v$-class group, nontrivial conductor, and  with all localizations being finitely primary.
It is well-known that the set of distances is finite, and this result holds without the seminormality assumption. However,  seminormality is crucial for the set of distances being an interval (even in the local case, sets of distances may fail to be intervals without assuming seminormality, see Examples \ref{3.4}).
There is an ideal-theoretic characterization  when the set of distances is empty  (\cite[Theorem 6.2]{Ge-Ka-Re15a}; a necessary condition  is that the $v$-class group has at most two elements). However, we did not want to include this characterization into the formulation of Theorem \ref{1.1}.

Suppose that the $v$-class group of $H$ is finite but not trivial and consider the assumption in Statement 2 of Theorem \ref{1.1}.
For $k \in \N$, let $\mathcal U_k (H)$ denote the set of all $\ell \in \N$ for which there is an equation of the form $u_1 \cdot \ldots \cdot u_k = v_1 \cdot \ldots \cdot v_{\ell}$ where $u_1, \ldots, u_k, v_1, \ldots, v_{\ell}$ are atoms (equivalently, $\mathcal U_k (H)$ is the union of all sets of lengths containing $k$). The map $\pi \colon \mathfrak X (\widehat H) \to \mathfrak X (H)$, defined by $\pi ( \mathfrak P) = \mathfrak P \cap H$ for all $\mathfrak P \in \mathfrak X (\widehat H)$, is surjective, and it is well-known that the unions $\mathcal U_k (H)$ are finite for all $k \in \N$ if and only if  $\pi$ is bijective. If the unions are finite and $H$ is seminormal, then the unions are finite intervals (\cite[Theorem 5.8]{Ge-Ka-Re15a}).  This shows that the first assumption in Statement 2 of Theorem \ref{1.1} is a natural one (our machinery is strong enough so that we can handle the slightly more general situation where there is at most one prime $\mathfrak p \in \mathfrak X (H)$ which is not inert in $\widehat H$).

Suppose that this assumption on the map $\pi$ does not hold. We settle the case where $G$ is an elementary $2$-group, and this allows us to show that $\Delta (H)=\emptyset$ or $\min \Delta (H)=1$ without any additional assumptions. The proof for elementary $2$-groups uses the fact that we know the maximum of the set of distances for Krull monoids whose class groups are elementary $2$-groups (this maximum is known only in very special cases; see Proposition \ref{2.4}).  We provide a detailed analysis of the case $|G|=2$ and determine the maximum of the set of distances (Theorem \ref{3.3}). In case of Krull monoids it is straightforward that the set of distances is empty if and only if $|G| \le 2$, and  Theorem \ref{3.3} reveals the complexity of the situation in the weakly Krull case.

The present paper is based on  ideal-theoretic results for $v$-noetherian weakly Krull monoids, recently established in \cite{Ge-Ka-Re15a}. They allow to study sets of distances in a special class of weakly Krull monoids which are easier to handle. Furthermore, we use that sets of distances in Krull monoids are intervals (\cite{Ge-Yu12b}). Our machinery will be put together in Section \ref{2}. The proof of Theorems \ref{1.1} and  \ref{3.3} will be given in Section \ref{3}, and we end with examples demonstrating the necessity of the various assumptions of Theorem \ref{1.1}.

\medskip
\section{Background in (weakly) Krull monoids} \label{2}
\medskip

We denote by $\N$ the set of positive integers, and for real numbers $a, b \in \R$ we denote by $[a, b] = \{ x \in \Z \mid a \le x \le b \}$ the discrete interval between $a $ and $b$. For subsets $A, B \subset \Z$, let $A + B = \{a+b \mid a \in A, b \in B\}$ be their sumset, and let $\Delta (A) = \{ d \in \N \mid d = l-k \ \text{for some} \ k, l \in L \ \text{with} \ [k, l] \cap A = \{k, l\} \}$ be the set of distances of $L$.

By a monoid, we mean a commutative cancellative semigroup with unit element. Let $H$ be a monoid. We denote by $\mathcal A (H)$ the set of atoms of $H$, by $H^{\times}$ the group of invertible elements of $H$, by $\mathsf q (H)$ the quotient group of $H$, and by $H_{\red} =H/H^{\times}$ the associated reduced monoid of $H$. For a set $\mathcal P$, we denote by $\mathcal F (\mathcal P)$ the \ {\it free
abelian monoid} \ with basis $\mathcal P$. Then every $a \in \mathcal F (\mathcal P)$ has a unique representation of the form
\[
a = \prod_{p \in \mathcal P} p^{\mathsf v_p(a) } \quad \text{with} \quad
\mathsf v_p(a) \in \N_0 \ \text{ and } \ \mathsf v_p(a) = 0 \ \text{
for almost all } \ p \in \mathcal P \,,
\]
and we  call $|a|_{\mathcal F (\mathcal P)} = |a|= \sum_{p \in \mathcal P}\mathsf v_p(a)$ the \emph{length} of $a$.
The  monoid  $\mathsf Z (H) = \mathcal F \bigl(
\mathcal A(H_\red)\bigr)$  is called the  {\it factorization
monoid}  of $H$, and  the unique homomorphism
\[
\pi \colon \mathsf Z (H) \to H_{\red} \quad \text{satisfying} \quad
\pi (u) = u \quad \text{for each} \quad u \in \mathcal A(H_\red)
\]
is  the  {\it factorization homomorphism}  of $H$. For $a
\in H$,
\[
\begin{aligned}
\mathsf Z_H (a) = \mathsf Z (a)  & = \pi^{-1} (aH^\times) \subset
\mathsf Z (H) \quad
\text{is the \ {\it set of factorizations} \ of \ $a$} \,,  \quad \text{and}
\\
\mathsf L_H (a) = \mathsf L (a) & = \bigl\{ |z| \, \bigm| \, z \in
\mathsf Z (a) \bigr\} \subset \N_0 \quad \text{is the \ {\it set of
lengths} \ of $a$}  \,.
\end{aligned}
\]
Thus $\mathsf L (a) = \{0\}$ if and only if $a \in H^{\times}$, and $\mathsf L (a) = \{1\}$ if and only if $a \in \mathcal A (H)$. The monoid $H$ is said to be {\it atomic} if $\mathsf Z (a) \ne \emptyset$ for every $a \in H$ (equivalently, every non-unit can be written as a finite product of atoms). If $H$ is $v$-noetherian (i.e., the ascending chain condition on divisorial ideals holds), then $H$ is atomic and all sets of lengths are finite and non-empty. Next, let
\[
\Delta (H) = \bigcup_{a \in H } \Delta \big(  \mathsf L (a) \big)
\]
denote the {\it set of distances} of $H$. Clearly, $\Delta (H)=\emptyset$ if and only if $|\mathsf L (a)|=1$ for each $a \in H$, and in this case $H$ is said to be half-factorial (for recent work on half-factorial domains see \cite{Co05a, Pl-Sc05a, Pl-Sc05b, Ma-Ok16a}).
If $H$ is not half-factorial, then $\min \Delta (H)= \gcd \Delta (H)$ (\cite[Proposition 1.4.4]{Ge-HK06a}). Thus, if there is an $m \in \N$ such that $m, m+1 \in \Delta (H)$, then $\min \Delta (H)=1$.

Let $z, z' \in \mathsf Z (H)$, say
\[
z= u_1 \cdot \ldots \cdot u_lv_1 \cdot \ldots \cdot v_m \quad \text{and} \quad z' = u_1 \cdot \ldots \cdot u_lw_1 \cdot \ldots \cdot w_n \,,
\]
where $l,m,n \in \N_0$ and $u_1, \ldots, u_l, v_1, \ldots, v_m,w_1, \ldots, w_n \in \mathcal A (H_{\red})$ with $\{v_1, \ldots, v_m\} \cap \{w_1, \ldots, w_n\} = \emptyset$. We call $\mathsf d (z,z')=\max \{m,n\} \in \N_0$ the {\it distance} between $z$ and $z'$. For every $N \in \N_0 \cup \{\infty\}$, an $N$-chain of factorizations of $a$ from $z$ to $z'$ is a finite sequence $(z_i)_{i \in [0,k]}$ of factorizations $z_i \in \mathsf Z (a)$ such that $z=z_0$, $z'=z_k$, and $\mathsf d (z_{i-1}, z_i) \le N$ for every $i \in [1,k]$.  For an element $a \in H$, its {\it catenary degree} $\mathsf c (a)$ is defined as the smallest $N \in \N_0 \cup \{\infty\}$ such that for any two factorizations $z, z' \in \mathsf Z (a)$ there is an $N$-chain of factorizations of $a$ from $z$ to $z'$. Then
\[
\mathsf c (H) = \sup \{ \mathsf c (a) \mid a \in H \} \in \N_0 \cup \{\infty\}
\]
denotes the {\it catenary degree} of $H$. If $a \in H$ has at least two distinct factorizations, then $2 + \sup \Delta \big( \mathsf L (a) \big) \le \mathsf c (a)$. The monoid $H$ is factorial if and only if it is atomic and $\mathsf c (H)=0$.  If $H$ is atomic but not factorial, then $2 + \sup \Delta (H) \le \mathsf c (H)$ (\cite[Theorem 1.6.3]{Ge-HK06a}).

Let $D$ be a monoid such that $H \subset D$ is a submonoid. We say that $H \subset D$ is
\begin{itemize}
\item {\it saturated} \ if $H = \mathsf q (H) \cap D$,
\item {\it divisor-closed} \ if $a \in H$, $\alpha \in D$, and $\alpha \t a$ imply that $\alpha \in H$,
\item {\it cofinal} \ if for every $\alpha \in D$ there is an $a \in H$ such that $\alpha \t a$, \ and
\end{itemize}
the factor group $\mathsf q (D)/\mathsf q (H)$ is called the {\it class group} of $H \subset D$.

Let $\mathfrak X (H)$ denote the set of all minimal non-empty prime $s$-ideals of $H$, and for subsets $A, B \subset \mathsf q (H)$, we set
$(A \DP B) = \{ x \in \mathsf q (H) \mid x B \subset A \}$.  We denote by  $\mathcal I_v^* (H)$  the monoid of $v$-invertible $v$-ideals (with $v$-multiplication) and by $\mathcal F_v (H)^{\times} = \mathsf q \big( \mathcal I_v^* (H) \big)$  its quotient group of fractional $v$-invertible $v$-ideals. The monoid of principal ideals $\mathcal H = \{aH \mid a \in H \}$ is a cofinal saturated submonoid of $\mathcal I_v^* (H)$, and the class group of $\mathcal H  \subset \mathcal I_v^* (H)$ is the $v$-class group $\mathcal C_v (H) = \mathcal F_v (H)^{\times}/ \mathsf q (\mathcal H)$ of $H$.  We denote by
\begin{itemize}
\item $H' = \{ x \in \mathsf q (H) \mid \ \text{there exists some} \ N \in
\N \
             \text{such that} \ x^n \in H \ \text{for all} \ n \ge N \}
$ the {\it seminormal closure} (also called the {\it
seminormalization}) of $H$, and by

\item $\widehat H = \{ x \in \mathsf q (H) \mid \ \text{there exists
             some } \ c \in H \ \text{such that} \ cx^n \in H \ \text{for all}
             \ n \in \N \}$ the {\it complete integral closure} of
             $H$ \,.
\end{itemize}
We say that $H$ is seminormal (completely integrally closed resp.) if $H = H'$ ($H = \widehat H$ resp.), and $(H \DP \widehat H)$  is called the conductor of $H$.
The localization $D_{\mathfrak p}$ of any monoid $D$ is a primary monoid for each $\mathfrak p \in \mathfrak X (D)$. We will mainly be concerned with a special class of primary monoids. A monoid $D$ is
called  {\it finitely primary} if there exist $s,\, \alpha \in \N$
such that $D$ is a submonoid of a factorial monoid $F= F^\times
\time [q_1,\ldots,q_s]$ with $s$ pairwise non-associated prime
elements $q_1, \ldots, q_s$ satisfying
\begin{equation} \label{finitelyprimary}
D \setminus D^\times \subset q_1 \cdot \ldots \cdot q_sF \quad
\text{and} \quad (q_1 \cdot \ldots \cdot q_s)^\alpha F \subset D
\,.
\end{equation}
If this holds, then $D$ is primary, $F = \widehat D$, $|\mathfrak X ( \widehat D)|=s$ is called the rank of $D$, and $D$ is seminormal if and only if
\[
D = q_1 \cdot \ldots \cdot q_s F \cup D^{\times} \,.
\]

\medskip
\begin{lemma} \label{2.1}
Let $D \subset F = F^\times \time [q_1,\ldots,q_s]$ be a seminormal
finitely primary monoid of rank $s$. Then
\begin{enumerate}
\item $\mathcal A (D) = \{ \epsilon q_1^{k_1} \cdot \ldots \cdot q_s^{k_s} \mid  \epsilon \in F^{\times} \  \text{\rm and} \
      \min \{k_1, \ldots, k_s \} = 1 \}$.

\smallskip
\item If $s=1$, then  $\mathsf c (D) \le 2$ and $D$ is half-factorial.

\smallskip
\item If $s \ge 2$, then $\min \mathsf L (a) = 2$ for all $a \in D \setminus ( D^{\times} \cup \mathcal A (D) )$. In particular,  $\mathsf c (D) = 3$.

\end{enumerate}
\end{lemma}

\begin{proof}
See \cite[Lemma 3.6]{Ge-Ka-Re15a}.
\end{proof}

A monoid $H$ is called a {\it weakly Krull monoid} (\cite[Corollary 22.5]{HK98}) if
\[
H = \bigcap_{{\mathfrak p} \in \mathfrak X (H)} H_{\mathfrak p}  \quad \text{and} \quad \{{\mathfrak p} \in \mathfrak X (H) \mid a \in {\mathfrak p}\} \quad \text{is finite for all} \ a \in H \,.
\]
A domain $R$ is weakly Krull if and only if its multiplicative monoid $R^{\bullet}=R\setminus \{0\}$ of nonzero elements is weakly Krull. Weakly Krull domains were introduced by Anderson, Anderson, Mott, and Zafrullah \cite{An-An-Za92b, An-Mo-Za92}, and a divisor theoretic characterization was first given by Halter-Koch \cite{HK95a}. For seminormal $v$-noetherian domains (i.e., seminormal Mori domains) we refer to the survey by Barucci \cite{Ba94}. The ideal theory of (general) weakly Krull monoids is presented in \cite[Chapters 21 -- 24]{HK98}, and for seminormal $v$-noetherian weakly Krull monoids we refer to \cite[Section 5]{Ge-Ka-Re15a}. A monoid $H$ is said to be {\it Krull} if it is weakly Krull and $H_{\mathfrak p}$ is a discrete valuation monoid for all $\mathfrak p \in \mathfrak X (H)$ (equivalently, $H$ is $v$-noetherian and completely integrally closed).

Let $H$ be a weakly Krull monoid. Then $H$ is $v$-noetherian (seminormal resp.) if and only if all localizations $H_{\mathfrak p}$ are $v$-noetherian (seminormal resp.) for each $\mathfrak p \in \mathfrak X (H)$.
Let $H$ be a seminormal $v$-noetherian weakly Krull monoid with $\emptyset \ne \mathfrak f = (H \DP \widehat H) \subsetneq H$. Then $\widehat H$ is Krull and for each $\mathfrak p \in \mathfrak X (H)$, $H_{\mathfrak p}$ is seminormal $v$-noetherian primary, and if
    $H$ is the multiplicative monoid of nonzero elements of a domain, then  $H_{\mathfrak p}$ is even finitely primary.

Noetherian domains are weakly Krull if and only if every prime ideal of depth one has height one, which holds in particular for all one-dimensional noetherian domains. Let  $R$  be a one-dimensional noetherian
domain such that its integral closure $\overline R$  is a finitely generated $R$-module.
Then the integral closure coincides with the complete integral closure, the conductor $\mathfrak f = (R \DP \overline R)$ is nonzero, and the $v$-class group is the usual Picard group.  If $R$ is an order in an algebraic number field or an order in a holomorphy ring of an algebraic function field, then the $v$-class group is finite and every class  contains infinitely many prime ideals. We refer to \cite{Ha04c,An-Ch-Pa06a, Ch09a, Li12a} for more on weakly Krull domains and to the  extended list of further examples  in \cite[Examples 5.7]{Ge-Ka-Re15a}.

\medskip
We continue with weakly Krull monoids of a combinatorial flavor which are used to model general weakly Krull monoids.
Let $G$ be an additive abelian group, $G_0 \subset G$ a
subset, $T$ a reduced monoid and $\iota \colon T \to G$ a homomorphism. Let
$\sigma \colon \mathcal F(G_0) \to G$ be the unique homomorphism
satisfying $\sigma (g) = g$ for all $g \in G_0$. Then
\[
B = \mathcal B (G_0,T,\iota) = \{S\,t \in \mathcal F(G_0) \time T\,\mid\, \sigma (S) + \iota(t) =
0\,\} \subset \mathcal F (G_0) \time T = F
\]
is called the  \textit{$T$-block monoid over  $G_0$  defined by $\iota$}\,.

\medskip
\begin{proposition} \label{2.2}
Let $D = \mathcal F (\mathcal P) \time T$ be a reduced atomic monoid, where $\mathcal P \subset D$ a set of primes
and $T \subset D$ is a  submonoid, and let $H \subset D$ be an atomic saturated submonoid with
class group $G = \mathsf q (D)/\mathsf q (H)$, and  $ G_{\mathcal P} = \{[p] \mid p \in \mathcal P\} \subset G$ the set of classes containing primes.
Let $\iota \colon T \to G$ be defined by $\iota (t) = [t]$, \ $F= \mathcal F (G_{\mathcal P}) \time T$, \ $B =
\mathcal B(G_{\mathcal P},T, \iota) \subset F$, and let $\widetilde{\boldsymbol \beta} \colon D \to F$ be the
unique homomorphism satisfying $\widetilde{\boldsymbol \beta} (p) = [p]$ for all $p \in \mathcal P$ and
$\widetilde{\boldsymbol \beta} \t T = \text{\rm id}_T$.

\begin{enumerate}
\item The restriction $\boldsymbol \beta = \widetilde{\boldsymbol
\beta} \t H \colon H \to B$ \ is a transfer homomorphism satisfying
$\mathsf c (H, \boldsymbol \beta) \le 2$. In particular, we have $\Delta (H) = \Delta (B)$ and $\mathsf c (H) = \mathsf c (B)$ (provided that $H$ is not factorial).

\smallskip
\item If $H \subset D$ is cofinal, then $B \subset F$ is cofinal, and there is an isomorphism $\overline \psi \colon \mathsf q (F)/\mathsf q (B) \to G$,
by which we will identify these groups.

\smallskip
\item If   $T=D_1 \times \ldots \times D_n$, $D_1, \ldots, D_n$ are seminormal finitely primary, and $G$ is a torsion group, then $H$ and $B$ are seminormal $v$-noetherian weakly Krull monoids with nontrivial conductors.
\end{enumerate}
\end{proposition}

\begin{proof}
See \cite[Proposition 3.4.8]{Ge-HK06a} and \cite[Lemma 5.2]{Ge-Ka-Re15a}.
\end{proof}

\smallskip
Next we consider monoids of zero-sum sequences which are well-studied submonoids of $T$-block monoids. As before, let $G$ be an additively written abelian group and $G_0 \subset G$ a subset.
In combinatorial number theory the elements of $\mathcal F (G_0)$ are called {\it sequences} over $G_0$ and
\[
\mathcal B (G_0) = \{ S \in \mathcal F (G_0) \mid \sigma (S) = 0 \} \subset \mathcal F (G_0)
\]
is the {\it monoid of zero-sum sequences} over $G_0$ (\cite{Ge-Ru09, Gr13a}). If $T$ and $B$ are as above, then $\mathcal B (G_0) \subset B$ is a divisor-closed submonoid whence $\mathsf Z_B (A) = \mathsf Z_{\mathcal B (G_0)} (A)$ and $\mathsf L_B (A) = \mathsf L_{\mathcal B (G_0)} (A)$ for all $A \in \mathcal B (G_0)$. If $T = \{1\}$, then
$\mathcal B (G_0) = B$. As usual, we set
\[
\mathcal A (G_0) := \mathcal A \big( \mathcal B (G_0) \big), \ \Delta (G_0) := \Delta \big( \mathcal B (G_0) \big) , \ \text{and} \ \mathsf c (G_0) := \mathsf c \big( \mathcal B (G_0) \big) \,.
\]
The atoms of $\mathcal B (G_0)$ are also called minimal zero-sum sequences over $G_0$. If $G_0$ is finite, then the set $\mathcal A (G_0)$ is finite, and
\[
\mathsf D (G_0) = \max \{ |U| \mid U \in \mathcal A (G_0) \}
\]
is the {\it Davenport constant} of $G_0$. Suppose that $G \cong C_{n_1} \oplus \ldots \oplus C_{n_r}$,  where $r, n_1, \ldots , n_r \in \N$ with $1 < n_1 \t \ldots \t n_r$, and set $\mathsf D^* (G) = 1 + \sum_{i=1}^r (n_i-1)$. Then $\mathsf D^* (G) \le \mathsf D (G)$, and equality holds for $p$-groups, groups of rank at most two, and others (\cite[Chapter 5]{Ge-HK06a}). It can be easily verified that $\Delta (G)= \emptyset$ if and only if $|G| \le 2$.

\medskip
\begin{proposition} \label{2.3}
Let $G$ be a finite abelian group with $|G| \ge 3$. Then $\Delta  (G)$ is a finite interval with $ \min \Delta  (G) = 1$.
\end{proposition}

\begin{proof}
See \cite[Theorem 1.1]{Ge-Yu12b}.
\end{proof}

\smallskip
Since $\mathcal B (G) \subset B = \mathcal B (G,T, \iota)$ is divisor-closed, $\Delta (B)$ contains the interval $\Delta (G) = [1, \max \Delta (G)]$. We will provide examples showing that in general (under the assumptions of Theorem \ref{1.1}) we have $\max \Delta (G) < \max \Delta (B)$ (see Theorem \ref{3.3} and Examples \ref{3.4}.4).
 The groups occurring in Statements 2 and 3 of Proposition \ref{2.4} are the only groups at all for which the precise value of $\max \Delta (G)$ is known.

\medskip
\begin{proposition} \label{2.4}
Let $G = C_{n_1} \oplus \ldots \oplus C_{n_r}$,  where $r, n_1, \ldots , n_r \in \N$ with $1 < n_1 \t \ldots \t n_r$, be a finite abelian group with $|G| \ge 3$.
\begin{enumerate}
\item \[
      \max \Big\{ \exp (G)-2, \sum_{i=1}^r \lfloor \frac{n_i}{2} \rfloor \Big\} \le \max \Delta  (G)  \le \mathsf c (G)-2 \le \mathsf D (G)-2 \,.
      \]

\item The following statements are equivalent{\rm \,:}
      \begin{enumerate}
      \item $\mathsf c (G)=\mathsf D (G)$.
      \item $\max \Delta  (G) = \mathsf D (G)-2$.
      \item $G$ is either cyclic or an elementary $2$-group.
      \end{enumerate}

\smallskip
\item The following statements are equivalent{\rm \,:}
      \begin{enumerate}
      \item $\mathsf c (G)=\mathsf D (G)-1$.
      \item $\max \Delta  (G) = \mathsf D (G)-3$.
      \item $G$ is isomorphic to $C_2^{r-1} \oplus C_4$ for some $r \ge 1$ or to $C_2 \oplus C_{2n}$ for some $n \ge 2$.
      \end{enumerate}
\end{enumerate}
\end{proposition}

\begin{proof}
1. See \cite[Theorem 6.7.1]{Ge-HK06a}.

2. The equivalence of (a) and (c) follows from \cite[Theorem 6.4.7]{Ge-HK06a} (and this is easy to prove). Statement 1. shows that (b) is equivalent to (a) and (c).

3. The equivalence of (a) and (c) follows from \cite[Theorem 1.1]{Ge-Zh15b} (and this requires some effort). Again Statement 1. shows that (b) is equivalent to (a) and (c).
\end{proof}

\medskip
\section{Arithmetic of weakly Krull monoids} \label{3}
\medskip

We fix our notation for the present section. Let
\[
B = \mathcal B (G, T, \iota) \subset F = \mathcal F (G) \time T \,,
\]
where $G$ is an additively written finite abelian group with $|G|>1$, $n \in \N$, $T= D_1 \time \ldots \time D_n$,  $D_1, \ldots, D_n$ are reduced seminormal finitely primary monoids, and  $\iota \colon T \to G$ be a homomorphism. Clearly,
\[
\mathcal A (F) = G \cup \bigcup_{i=1}^n \mathcal A (D_i) \,,
\]
Since $G$ is finite, $B \subset F$ is a cofinal  saturated submonoid with class group $\mathsf q (F)/\mathsf q (B) \cong  G$ (\cite[Proposition 3.4.7]{Ge-HK06a}), and we identify the groups. For every $a \in \mathsf q (F)$, we denote by $[a] = a \mathsf q (B)$ the class containing $a$, and since $B \subset F$ is saturated,  we have $a \in B$ if and only if $[a] = 0 \in G$. In particular, we have $a^{\exp (G)} \in B$ for every $a \in F$.
For every $i \in [1,n]$, we  set $D_i \subset \widehat{D_i} = \widehat{D_i}^{\times} \time [q_{i,1}, \ldots, q_{i, s_i}]$ with $s_i \in \N$, and we have $\mathcal A (D_i) = \{\epsilon q_{i,1}^{k_1} \cdot \ldots \cdot q_{i, s_i}^{k_{s_i}}  \mid \epsilon \in \widehat{D_i}^{\times}, \min \{k_1, \ldots, k_{s_i}\}=1  \}$ by Lemma \ref{2.1}. Every $A \in F$ has a unique product decomposition of the form
\[
A = g_1 \cdot \ldots \cdot g_k a_1 \cdot \ldots \cdot a_n \,,
\]
where $k \in \N_0$, $g_1, \ldots, g_k \in G$, and $a_i \in D_i$ for every $i \in [1,n]$. For the set of factorizations of $A$ we observe that
\[
\mathsf Z_F (A) = g_1 \cdot \ldots \cdot g_k \prod_{i=1}^n \mathsf Z_F (a_i) \,.
\]
We define a norm
\[
\begin{aligned}
\parallel \cdot \parallel  \colon \mathcal F (G) \times T & \quad \to \quad (\N_0, +)  \\
  A = g_1 \cdot \ldots \cdot g_k a_1 \cdot \ldots \cdot a_n & \quad \mapsto \quad k + 2 \sum_{i=1}^n \max \mathsf L_{D_i}(a_i) = k+2 \max \mathsf L_{F} (a_1 \cdot \ldots \cdot a_n) \,.
\end{aligned}
\]
Obviously,  $\norm{S} = |S|_{\mathcal F (G)}$ for all $S \in \mathcal F (G)$,  and $\norm{A} = 0$ if and only if $A=1 \in F$.
For each $i \in [1,n]$, let  $\mathsf p_i \colon \mathcal F (G) \times D_1 \times \ldots \times D_n \to D_i$ denote the projection.

\medskip
\begin{lemma}  \label{3.1}
Let $B$ be as above and $A \in B$ be an atom.
\begin{enumerate}
\item If $q$ is atom of $T$ such that  $q\t A $ (in $F$) and  $g= [q] \in G$,  then $gq^{-1}A$ is also an atom of  $B$.

\smallskip
\item  Let $i \in [1,n]$, $\epsilon\in \widehat{D_i}^{\times}$, and $q =\epsilon q_{i,1} \cdot \ldots \cdot q_{i, s_i}$ be an atom of $T$ such that  $g\t A$ where $g= [q] \in G$. Then $qg^{-1}A$ is either an atom of  $B$ or a product of two atoms of  $B$.
\end{enumerate}

\end{lemma}

\begin{proof}
1. Obvious.

\smallskip
2. If $\mathsf p_i(A) =1$, then $qg^{-1}A$ is obviously  an atom.
Suppose that $\mathsf p_i(A)=\epsilon_1 q_{i,1}^{k_1} \cdot \ldots \cdot q_{i, s_i}^{k_{s_i}}$ with $ \epsilon_1 \in \widehat{D_i}^{\times}$ and $k_j\ge 1$ for all $j\in [1,s_i]$. Let $A'=qg^{-1}A=W_1 \cdot \ldots \cdot W_t$ with $W_i\in \mathcal A(B)$ for all $i \in [1,t]$. Assume to the contrary that  $t\ge 3$. If $\mathsf p_i(W_1)=1$ or $\mathsf p_i(W_2)=1$, then $W_1\t A$ or $W_2\t A$, a contradiction.  Otherwise $\mathsf p_i(W_1)\neq 1$ and $\mathsf p_i(W_2)\neq 1$ which implies that $\mathsf p_i(W_3)=\epsilon_2 q_{i,1}^{r_1} \cdot \ldots \cdot q_{i, s_i}^{r_{s_i}}$ with $ \epsilon_2 \in \widehat{D_i}^{\times}$ and $r_j\le k_j-1$ for all $j\in [1,s_i]$. Hence $\mathsf p_i(W_3)\t \mathsf p_i(A)$ and $W_3\t A$, a contradiction.
\end{proof}

\medskip
\begin{proposition}\label{3.2}
Let $B$ be as above and $A \in B$ with $\max \Delta \big( \mathsf L (A) \big) \ge \max \Delta(G)+2$.
Suppose that  $|\{i \in [1,n] \mid s_i >1 \}| \le 1$ or that $G$ is an elementary $2$-group.
Then there exists an  $A'\in B$ with $\norm{A'} < \norm{A}$ and $\max \Delta \big( \mathsf L(A') \big) \ge \max \Delta \big( \mathsf L (A) \big) -1$.
\end{proposition}

\begin{proof}
Suppose that
\[
A=U_1\cdot \ldots \cdot U_k=V_1\cdot \ldots \cdot V_l \,,
\]
where $U_1, \ldots U_k,V_1,\ldots,V_l\in \mathcal A(B)$, $[k+1,l-1]\cap \mathsf L_B (A)=\emptyset$, and $l-k=\max \Delta \big( \mathsf L (A) \big)  \ge \max \Delta(G)+2\ge \exp(G) $ (for the last inequality we use Proposition \ref{2.4}.1). We  distinguish several cases.

\smallskip
\noindent
CASE 1: \, There exist $i \in [1,k]$ and  $g_1,g_2\in G$ such that $g_1g_2 \t U_i $, say $i=1$.

Let $U_1'=U_1(g_1g_2)^{-1}(g_1+g_2)$ and $A'=AU_1^{-1}U_1'$. Then $\norm{A'}  < \norm{ A }$, $U_1'$ is also an atom of $B$, and $k\in \mathsf L_B(A')$. After renumbering if necessary
 we may assume that $g_1g_2\t V_1V_2$. Let $V'=V_1V_2(g_1g_2)^{-1}(g_1+g_2)$. Then $A'=V'V_3\cdot\ldots \cdot V_l$ and there exists $m_0\ge l-1>k$ such that $m_0\in \mathsf L_B(A')$. Choose $m=\min \big(\mathsf L_B(A')\setminus [1,k] \big)$. We only need to prove that $m\ge l-1$. Assume to the contrary that $k<m<l-1$.
 Then let $A'=W_1\cdot\ldots\cdot W_m$ with $W_i\in \mathcal A(B)$ for all $i\in [1,m]$ and $g_1+g_2\t W_1$. Let $W_1'=W_1(g_1+g_2)^{-1}g_1g_2$. Then $A=W_1'W_2\cdot\ldots\cdot W_m$ and $W_1'$ is an atom or a product of two atoms. Hence $m$ or $m+1 \in \mathsf L_B(A)$, a contradiction to $[k+1,l-1]\cap \mathsf L_B (A)=\emptyset$.

\smallskip
\noindent CASE 2: \, There exist $i \in [1,n]$, $\epsilon\in \widehat{D_i}^{\times}$, and $j \in [1,k]$ such that  $q=\epsilon q_{i,1} \cdot \ldots \cdot q_{i, s_i}$   divides  $U_j $, say $i=j=1$.

Let $U_1'=U_1q^{-1} [q]$ and $A'=AU_1^{-1}U_1'$. Then $\norm{A'}  < \norm{ A }$, $U_1'$ is  an atom of $B$ by Lemma \ref{3.1}.1, and $k\in \mathsf L_B(A')$. After renumbering if necessary
 we may assume that $q\t V_1V_2$. Let $V'=V_1V_2q^{-1} [q]$. Then $A'=V'V_3\cdot\ldots \cdot V_l$ and there exists $m_0\ge l-1>k$ such that $m_0\in \mathsf L_B(A')$. Choose $m= \min \big( \mathsf L_B(A')\setminus [1,k] \big)$. We only need to prove that $m\ge l-1$. Assume to the contrary that $k<m<l-1$.
 Then let $A'=W_1\cdot\ldots\cdot W_m$ with $W_i\in \mathcal A(B)$ for all $i\in [1,m]$ and $[q] \t W_1$. Let $W_1'=W_1 [q]^{-1}q$. Then $A=W_1'W_2\cdot\ldots\cdot W_m$ and $W_1'$ is an atom or a product of two atoms by Lemma \ref{3.1}.2. Hence $m$ or $m+1 \in \mathsf L_B(A)$, a contradiction to $[k+1,l-1]\cap \mathsf L_B (A)=\emptyset$.

\smallskip
\noindent CASE 3: \, There exists an $i \in [1,n]$ such that $\mathsf p_i(A)\neq 1$ and $\max \mathsf L_{F}\big(\mathsf p_i(A)\big)\le 2$.

Without loss of generality, we assume that $\mathsf p_i(U_1)\neq 1$. Let $U_1'=U_1\mathsf p_i(U_1)^{-1}[\mathsf p_i(U_1)]$ and   $A'=A U_1^{-1} U_1'$. Then  $\norm{A'} < \norm{A}$, $U_1'$ is also an atom of $B$, and $k\in \mathsf L_B(A')$.  Since $\max \mathsf L_{F}\big(\mathsf p_i(A)\big)\le 2$, we may assume, after renumbering if necessary,  that $\mathsf p_i(U_1)\t\mathsf p_i(A)\t V_1V_2$. Let $V'=V_1V_2\mathsf p_i(U_1)^{-1}[\mathsf p_i(U_1)]$. Then $A'=V'V_3 \cdot \ldots \cdot V_l$ and  there exists $m_0\ge l-1> k$ such that $m_0\in \mathsf L_B(A')$. We suppose that  $A'=W_1\cdot \ldots\cdot W_{m}$ with $[\mathsf p_i(U_1)]\t W_1$ and $m= \min \big( \mathsf L_B(A')\setminus [1,k] \big)$, where $W_i\in \mathcal A(B)$ for each $i\in [1,m]$. Let $W_1'=W_1[\mathsf p_i(U_1)]^{-1}\mathsf p_i(U_1)$ and hence $W_1'$ is an atom or a product of two atoms by $\max \mathsf L_{F}\big(\mathsf p_i(A)\big)\le 2$. Since $A=W_1'W_2 \cdot \ldots \cdot W_m$, it follows that  $m\ge l-1$, whence $\max \Delta \big( \mathsf L (A') \big)   \ge m-k\ge l-1-k= \max \Delta \big( \mathsf L (A) \big) -1$.

\smallskip
We summarize what we know so far. If $A \in \mathcal F (G)$, then CASE 1 holds. After renumbering and replacing $n$ by some $n' \in [1,n]$ if necessary we may suppose that $\mathsf p_i (A) \ne 1$ for each $i \in [1,n]$. By CASE 3, we may suppose that $\max \mathsf L_F\big(\mathsf p_i (A)\big)\ge 3$ for each $i \in [1,n]$.  If there is some $i \in [1,n]$ with $s_i=1$, then CASE 2 holds (see Lemma \ref{2.1}). Thus we may suppose that $s_i > 1$ for each $i \in [1,n]$. Then the inequality $|\{i \in [1,n] \mid s_i > 1 \}| \le 1$ made in the assumption of the proposition implies that $n=1$. Again by CASE 2, we infer that $\mathsf p_i (U_j) \in \mathcal A (D_i) \cup \{1\}$ for each $i \in [1,n]$ and each $j \in [1,k]$.
Now we continue with further case distinctions.

\smallskip
\noindent CASE 4: \, $k\ge 3$.

After renumbering if necessary we may suppose that $\mathsf p_1 (U_1) \in \mathcal A (D_1)$ and $\mathsf p_1 (U_2) \in \mathcal A (D_1)$.
If $\min \big( \mathsf L_B(U_1U_2)\setminus\{2\} \big) \ge l-k+2$, then    $A'=U_1U_2$ satisfies  $\norm{A'}  < \norm{ A} $ and $\max \Delta \big( \mathsf L (A') \big)  \ge \max \Delta \big( \mathsf L (A) \big) -1$.
If  $r = \min \big( \mathsf L_B(U_1U_2)\setminus\{2\} \big) \in [3, l-k+1]$, say
$U_1U_2=W_1\cdot\ldots\cdot W_r$ with atoms $W_1, \ldots, W_r$, then
$A=W_1\cdot\ldots\cdot W_rU_3\cdot\ldots\cdot U_k$ and $k+1\le r+k-2\le l-1$, a contradiction to $[k+1,l-1]\cap \mathsf L_B (A)=\emptyset$.

Now suppose that $ \mathsf L_B(U_1U_2)=\{2\}$. Since $\mathsf p_1(U_1)\neq 1$ and $\mathsf p_1(U_2)\neq 1$, we infer that  $q=q_{1,1} \cdot \ldots \cdot q_{1, s_1}\t \mathsf p_1(U_1U_2)$, and we set $U'=U_1U_2q^{-1} [q]$ and $A'=U'U_3\cdot\ldots\cdot U_k$. Clearly, we have  $\norm{A'} < \norm{ A }$.

Suppose that $U'=W_1\cdot \ldots \cdot W_t$ with  $t\ge 3$ and atoms $W_1, \ldots, W_t$. After renumbering if necessary we suppose that $[q] \t W_1$, and then  $W_1'=W_1 [q]^{-1}q$ is an atom or a product of two atoms by Lemma \ref{3.1}.2. Therefore $U_1U_2=W_1'W_2\cdot\ldots\cdot W_t$ and $t$ or $t+1\in \mathsf L_B(U_1U_2)$, a contradiction. Hence $U'$ is an atom or a product of two atoms which implies that  $k$ or $k-1\in \mathsf L_B(A')$.

After renumbering if necessary,
  we may assume that $q\t V_1V_2$, and we set $V'=V_1V_2q^{-1} [q]$. Then $A'=V'V_3\cdot\ldots \cdot V_l$ and there exists $m_0\ge l-1>k$ such that $m_0\in \mathsf L_B(A')$. We choose $m=\min \big( \mathsf L_B(A')\setminus [1,k] \big)$ and need to prove that $m\ge l-1$. Assume to the contrary that $k<m<l-1$.
  Then $A'=W_1\cdot\ldots\cdot W_m$ with $W_i\in \mathcal A(B)$ for all $i\in [1,m]$ and $[q] \t W_1$. Let $W_1'=W_1 [q]^{-1}q$. Then $A=W_1'W_2\cdot\ldots\cdot W_m$ and $W_1'$ is an atom or a product of two atoms by Lemma \ref{3.1}.2. Hence $m$ or $m+1 \in \mathsf L_B(A)$, a contradiction to $[k+1,l-1]\cap \mathsf L_B (A)=\emptyset$.

\smallskip
\noindent CASE 5: \, $k=2$.

We suppose that none of the previous cases holds.  Therefore, after a suitable renumbering if necessary, we have    $A=U_1U_2$ with  $U_1=S_1 a_1\cdot\ldots\cdot a_n \in \mathcal A (B)$, $U_2=S_2 b_1\cdot \ldots\cdot b_n \in \mathcal A (B)$, where
$S_1,S_2 \in \mathcal F (G)$ with $|S_1| \le 1$, $|S_2| \le 1$, and for each $i\in [1,n]$,
\begin{align*}
&a_i=\epsilon_iq_{i,1}^{k_{i,1}} \cdot \ldots \cdot q_{i,s_i}^{k_{i,s_i}} \ \in D_i \ \text{ with } k_{i,1}=1 \text{ and }k_{i,s_i}>1\,,\\
&b_i=\epsilon_i'q_{i,1}^{t_{i,1}} \cdot \ldots \cdot q_{i,s_i}^{t_{i,s_i}} \ \in D_i \ \text{ with }t_{i,1}>1 \text{ and } t_{i,s_i}=1\,, \text{ where $\epsilon_i, \epsilon_i'\in \widehat{D_i}^{\times}$ }.
\end{align*}

\medskip
First, we suppose that there exists $i\in [1,n]$   such that $\max \mathsf L_F(a_ib_i)\ge \exp(G)+2$, say $i=n$.

After renumbering if necessary there is a $\lambda\in [1,s_n]$ such that
$k_{n,1}=\ldots=k_{n,\lambda}=1$   and $k_{n,\rho}>1$ for all $\rho\in [\lambda+1, s_n]$. Then $t_{n,\rho}\ge \exp(G)+1$ for each $\rho\in [1, \lambda]$.
It follows that $U_1U_2=U_1'U_2'$ where
\begin{align*}
& U_1'=U_1a_n^{-1}\left(\epsilon_iq_{n,1}^{1+\exp(G)} \cdot \ldots \cdot q_{n,\lambda}^{1+\exp(G)}q_{n,\lambda+1}^{k_{n,\lambda+1}} \cdot \ldots \cdot q_{n,s_n}^{k_{n,s_n}}\right) \in B \text{ and }\\
 & U_2'=U_2b_n^{-1}\left(\epsilon_i'q_{n,1}^{t_{n,1}-\exp(G)} \cdot \ldots \cdot q_{n,\lambda}^{t_{n,\lambda}-\exp(G)}q_{n,\lambda+1}^{t_{n,\lambda+1}} \cdot \ldots \cdot q_{n,s_n}^{t_{n,s_n}}\right) \in \mathcal A (B) \,.
\end{align*}
If $U_1' \in \mathcal A (B)$, then $q_{n,1} \cdot \ldots \cdot q_{n,s_n}\t U_1'$ and hence the assumption of   {CASE 2} is satisfied.  Otherwise, set $\ell=\min \mathsf L_B (U_1')$ and hence $\ell\in [2,\exp(G)+1]$. If $\ell\in [2, \exp(G)]$, then $\max \Delta \big( \mathsf L (A) \big)  = l-2\le\ell+1-2<\exp(G)$, a contradiction. Thus $\ell=\exp(G)+1$, $\max \Delta \big( \mathsf L (A) \big) =\exp(G)$, and hence $k_{n,\rho}\ge \exp(G)+1$ for all $\rho\in [\lambda+1, s_n]$. Since
\[
U_1 a_n^{-1}\left(\epsilon_iq_{n,1} \cdot \ldots \cdot q_{n,\lambda}q_{n,\lambda+1}^{k_{n,\lambda+1}-\exp(G)} \cdot \ldots \cdot q_{n,s_n}^{k_{n,s_n}-\exp(G)}\right) \in \mathcal A (B) \,,
\]
we obtain that $\mathsf L_B \left(q_{n,1}^{\exp(G)} \cdot \ldots \cdot q_{n,s_n}^{\exp(G)}\right)=\{\exp(G)\}$. We set
\[
A'=q_{n,1}^{\exp(G)+1} \cdot \ldots \cdot q_{n,s_n-1}^{\exp(G)+1}q_{n,s_n}^{\exp(G)+2}\cdot (-[q_{n,s_n}])
\]
and observe that $\norm{A'} < \norm{A}$. Since $\mathsf L_B (A')=\{2, \exp(G)+1\}$, it follows that $\max \Delta \big( \mathsf L (A') \big) =\exp(G)-1  \ge  \max \Delta  \big( \mathsf L (A) \big) -1=\exp(G)-1$.

\medskip
From now on we suppose that  $\max \mathsf L_F(a_ib_i)\le \exp(G)+1$ for each $i\in [1,n]$, and distinguish two cases.

\medskip
\noindent
CASE 5.1: \ $n= 1$ (recall all the reductions made before CASE 4).

 Then $U_1=S_1a_1$ and $U_2=S_2b_1$ with $|S_1|\le 1$ and $|S_2|\le 1$.
  Since $\max \Delta \big( \mathsf L (A) \big) \ge \exp(G)$, we have that $\min \big( \mathsf L_B (A)\setminus\{2\} \big) \ge \exp(G)+2$. By $\max \mathsf L_F(a_1b_1)\le \exp(G)+1$, we obtain that $|S_1=|S_2|=1$, $S_1S_2 \in \mathcal A (B)$, $a_1a_2 \in B$,  $\min \big( \mathsf L_B (A)\setminus\{2\} \big) = \exp(G)+2$,  $\max \mathsf L_F(a_1b_1)= \exp(G)+1$ and hence $\mathsf L_B(A)=\{2, \exp(G)+2\}$. Since
  $q_{1,1}^{\exp(G)} \cdot \ldots \cdot q_{1,s_1}^{\exp(G)}\t a_1b_1$, we have that  $a_1b_1(q_{1,1}^{\exp(G)} \cdot \ldots \cdot q_{1,s_1}^{\exp(G)})^{-1}$ is an atom and  $\mathsf L_B \left(q_{1,1}^{\exp(G)} \cdot \ldots \cdot q_{1,s_1}^{\exp(G)}\right)=\{\exp(G)\}$.

 We set
\[
A'=q_{1,1}^{\exp(G)+1} \cdot \ldots \cdot q_{1,s_1-1}^{\exp(G)+1}q_{1,s_1}^{\exp(G)+2}\cdot (-[q_{1,s_1}])
\]
and observe that $\norm{A'} = 1+2(\exp (G)+1) < 2+2(\exp (G)+1)= \norm{A}$. Since $\mathsf L_B (A')=\{2, \exp(G)+1\}$, it follows that $\max \Delta \big( \mathsf L (A') \big)=\exp(G)-1\ge\max\Delta \big( \mathsf L (A) \big)-1=\exp(G)-1$.

\medskip
\noindent
CASE 5.2: \ $G$ is an elementary $2$-group, say $G \cong C_2^r$.

We may assume that $n \ge 2$.
Since $l-2=\max \Delta \big( \mathsf L (A) \big) \ge \max \Delta(C_2^r)+2=r+1$ (where the last equation follows from Proposition \ref{2.4}), we have that $l\ge r+3\ge 4$.

Since $3\le \max \mathsf L_F (\mathsf p_i(A))=\max \mathsf L_F (a_ib_i)\le \exp(G)+1=3$ for each $i\in [1,n]$, we may assume  that
 \begin{align*}
&a_n=\epsilon_nq_{n,1}q_{n,2}^{k_{n,2}} \cdot \ldots \cdot q_{n,s_n}^{k_{n,s_n}} \text{ with }  k_{n,2}>1 \quad \text{and} \,,\\
&b_n=\epsilon_n'q_{n,1}^2q_{n,2}q_{n,3}^{t_{n,3}} \cdot \ldots \cdot q_{n,s_n}^{t_{n,s_n}} \  \text{ where $\epsilon_n, \epsilon_n'\in \widehat{D_n}^{\times}$ and  } k_{n,j}+t_{n,j}\ge 3 \text{ for each $j\in [3,n]$}\,.
\end{align*}

\medskip
\noindent
CASE 5.2.1: \ $l=4$.

Then $r=1$ and $G\cong C_2$. If $[q_{n,1}]=e$, $[q_{n,k}]=0$ for each $k\in [2,s_n]$, and $\{ [\eta] \mid \eta \in \widehat{D_n}^{\times}\}=\{0\}$, then $U_2=b_n$ and hence $n=1$, a contradiction. Thus we may suppose  that $[q_{n,1}]=[q_{n,k}]$ for some $k\in[2,s_n]$ or $[q_{n,1}]=[\eta_0]$  for some $\eta_0\in \widehat{D_n}^{\times}$ because $[1_{\widehat{D_n}^{\times}}]=0$. Let
 \begin{align*}
 b_n'=\epsilon_n'\eta q_{n,1}q_{n,2}q_{n,3}^{t_{n,3}} \cdot \ldots \cdot q_{n,s_n}^{t_{n,s_n}}, \text{ where } \eta=\left\{\begin{aligned}
 &q_{n,k} &&\text{ if $[q_{n,1}]=[q_{n,k}]$ for some $k\in [2,s_n]$,}\\
 &\eta_0&& \text{ otherwise } [q_{n,1}]=[\eta_0]  \text{ for some } \eta_0\in \widehat{D_n}^{\times}.\\
 \end{aligned}\right.\end{align*}
  and hence $U_2'=U_2b_n^{-1}b_n'$ is an atom. We set $A'=U_1U_2'$ and observe that $\norm{A'}<\norm{A}$. After renumbering if necessary we may  assume that $q_{n,1} \cdot \ldots \cdot q_{n,s_n}\t V_1V_2$. Let
  $$\mathsf p_n(V_1V_2)=\epsilon_1q_{n,1}^{x_1}q_{n,2}^{x_2} \cdot \ldots \cdot q_{n,s_n}^{x_{s_n}} \text{ with }  x_{k}>1 \text{ for each $k\in [1,s_n]$}$$  and
\[
d=\epsilon_1\eta q_{n,1}^{x_1-1}q_{n,2}^{x_2} \cdot \ldots \cdot q_{n,s_n}^{x_{s_n}} \,.
\]
Then
\[
A'=V_1V_2\mathsf p_n(V_1V_2)^{-1} d  V_3  V_4 \,,
\]
which implies that there is a $k\in \N$ such that $k\ge 3$ and $k\in \mathsf L_B (A')$. Then $\max \Delta \big( \mathsf L (A') \big)\ge 1=l-3=\max\Delta \big( \mathsf L (A) \big)-1$.

\medskip
\noindent
CASE 5.2.2: \  $l\ge 5$.

Then $U_1'=U_1a_n^{-1}[a_n] \in \mathcal A (B)$, and after renumbering if necessary we may  assume that $a_n\t V_1V_2V_3$. We set
\[
A'=U_1'U_2=V_1V_2V_3a_n^{-1}[a_n]\cdot V_4\cdot\ldots\cdot V_l \,,
\]
and observe that $\norm{A'}<\norm{A}$ and there is a $k\in \N$ such that $k\ge l-2\ge 3$ and $k\in \mathsf L_B (A')$.
 Suppose that $m=\min \big( \mathsf L_B (A')\setminus \{2\} \big)$ and  $A'=W_1\cdot \ldots \cdot W_{m}$ with $[a_n]\t W_1$, where $W_i$ is an atom for each $i\in [1,m]$.

 If $b_n\nmid W_1$, then  $W_1'=W_1[a_n]^{-1}a_n$ is an atom and $U_1U_2=W_1'W_2 \cdot \ldots \cdot  W_{m}$. Therefore $m\ge l$ which implies that $\max \Delta \big( \mathsf L (A') \big)\ge \max \Delta \big( \mathsf L (A) \big)$.  Thus we assume that $b_n\t W_1$. Let $W_1'=W_1[a_n]^{-1}a_n$ and hence $q=q_{n,1}^2q_{n,2}^2 \cdot \ldots \cdot q_{n,s_n}^2\t a_nb_n\t W_1'$. Therefore $W_1'q^{-1}$ is an atom and $q$ is an atom or a product of two atoms. If $q$ is an atom, then $m+1\ge l$ and hence $\max \Delta \big( \mathsf L (A') \big)\ge \max \Delta \big( \mathsf L (A) \big)-1$.

Suppose that $q$ is a product of two atoms. Then there exists $\epsilon\in \widehat{D_n}^{\times}$ such that $\epsilon q_{n,1}q_{n,2} \cdot \ldots \cdot q_{n,s_n} $ is an atom. Let $b_n'=\epsilon_n'\epsilon q_{n,1}q_{n,2}^2q_{n,3}^{t_{n,3}+1} \cdot \ldots \cdot q_{n,s_n}^{t_{n,s_n}+1}$ and hence $U_2'=U_2b_n^{-1}b_n'$ is an atom.
 Without loss of generality, we assume that $b_n\t V_1V_2V_3$. Therefore
 \[
 A''=U_1U_2'=V_1V_2V_3b_n^{-1}b_n' V_4 \cdot \ldots \cdot V_l \,,
 \]
 which implies that $\norm{A''}<\norm{A}$ and there is a $k\in \N$ such that $k\ge l-2\ge 3$ and $k\in \mathsf L_B (A'')$.

  Suppose that $m=\min \big( \mathsf L_B (A'')\setminus \{2\} \big)$ and   $A''=X_1 \cdot \ldots \cdot X_{m}$ with  $X_i$ is an atom for each $i\in [1,m]$.
If $\mathsf p_n(A'')=\mathsf p_n(X_i)$ for some $i\in [1,m]$, then $\epsilon q_{n,1} \cdot \ldots \cdot q_{n,s_n}\t a_nb_n'\t X_i$, a contradiction to that $\epsilon q_{n,1} \cdot \ldots \cdot q_{n,s_n} \in \mathcal A (B)$. Therefore we may assume that $\mathsf p_n(A'')=\mathsf p_n(X_1X_2)$, $\mathsf p_n(X_1)\neq 1$, and that $\mathsf p_n(X_2)\neq 1$. Let
\[
\mathsf p_n(X_1) =\epsilon_1 q_{n,1}q_{n,2}^{r_2} \cdot \ldots \cdot q_{n,s_n}^{r_n} \quad \text{and} \quad
\mathsf p_n(X_2) =\epsilon_2 q_{n,1}q_{n,2}^{s_2} \cdot \ldots \cdot q_{n,s_n}^{s_n}
\]
with $r_i+s_i\ge 4$ for each $i\in [2,n]$. Let
\begin{align*}
\mathsf c_n&=\epsilon_1 q_{n,1}q_{n,2}^{r_2'} \cdot \ldots \cdot q_{n,s_n}^{r_n'}\,, \text{ with }
 r_i'=\left\{
 \begin{aligned}
& r_i &\text{ if $r_i\le 2$ ,}\\
& r_i-2 &\text{ if $r_i\ge 3$ ,}
 \end{aligned}\right. \text{ for each $i\in [2,n]$}\,,\\
\mathsf d_n&=\epsilon_2 \epsilon^{-1} q_{n,1}^2q_{n,2}^{s_2'} \cdot \ldots \cdot q_{n,s_n}^{s_n'}\,,
\text{ with }
 s_i'=\left\{
 \begin{aligned}
& s_i-1 &\text{ if $r_i\le 2$ ,}\\
& s_i+1 &\text{ if $r_i\ge 3$ ,}
 \end{aligned}\right. \text{ for each $i\in [2,n]$}\,.
\end{align*}
Therefore $c_nd_n=a_nb_n$. Let $X_1'=X_1 \mathsf p_n(X_1)^{-1}c_n$ and $X_2'=X_2 \mathsf p_n(X_2)^{-1}d_n$. Then $X_1'$ is an atom, $X_2$ is an atom or a product of two atoms, and $A=X_1'X_2'X_3 \cdot \ldots \cdot X_{m}$ which implies that $m\ge l-1$ and hence $\max \Delta \big( \mathsf L (A'') \big)\ge m-2\ge l-3=\max\Delta \big( \mathsf L (A) \big)-1$.
\end{proof}

\medskip
Recall that for an atomic but non-factorial monoid $H$ we have $2 + \sup \Delta (H) \le \mathsf c (H)$. In general, this inequality can be strict (even for numerical monoids;  see Examples \ref{3.4}.2).
Suppose $H$ is a Krull monoid with finite class group $G$ and suppose that every class contains a minimal prime ideal. Then $\Delta (H) = \emptyset$ if and only if $|G| \le 2$, and if $G$ is nontrivial with $\mathsf D (G) = \mathsf D^* (G)$, then $2 + \max \Delta (H) = \mathsf c (H)$ (\cite[Corollary 4.1]{Ge-Gr-Sc11a}).

Now let
$H$ be a weakly Krull monoid as in Theorem \ref{1.1} but not Krull, whence $H$ is seminormal $v$-noetherian with nontrivial conductor, all localizations $H_{\mathfrak p}$ are finitely primary, and every class of the $v$-class group contains a minimal prime ideal $\mathfrak p \in \mathfrak X (H)$ with $\mathfrak p \not\supset \mathfrak f$.  If $G$ is trivial, then $H_{\red}$ is isomorphic to the monoid $\mathcal I_v^* (H)$ of $v$-invertible $v$-ideals and $2 + \max \Delta (H) = \mathsf c (H) \in \{2,3\}$ (this will be outlined in detail in the proof of Theorem \ref{1.1}). The next theorem provides a detailed analysis of the case where $|G|=2$.

\medskip
\begin{theorem} \label{3.3}
Let $H$ be a seminormal $v$-noetherian weakly Krull monoid, $\widehat H$ its complete integral closure,  $\emptyset \ne \mathfrak f = (H \DP \widehat H) \subsetneq H$ its conductor,
 $\mathcal P^* = \{\mathfrak p_1, \ldots, \mathfrak p_n\} = \{\mathfrak p \in \mathfrak X (H) \mid \mathfrak p \supset \mathfrak f \}$, and $\mathcal P = \mathfrak X (H) \setminus \mathcal P^*$.
 Suppose that  $H_{\mathfrak p_i}$ is finitely primary of rank $s_i$ for all $i \in [1,n]$,   that the $v$-class group $G=\mathcal C_v (H)$ has two elements, and that each class contains some $\mathfrak p \in \mathcal P$.

Then $\Delta (H)$ is an interval with $2+\max \Delta (H) = \mathsf c (H)$, and either $\Delta (H)=\emptyset$ or $\min \Delta (H)=1$. Moreover,  setting $G = \{0, e\}$,
\[
D_{\nu} = (H_{\mathfrak p_{\nu}})_{\red} \,, \quad  \widehat{D_{\nu}} = \widehat{D_{\nu}}^{\times} \time [q_{{\nu},1}, \ldots, q_{{\nu}, s_{\nu}}]  \,, \quad \text{and} \quad G_{\nu} = \{[\epsilon] \mid \epsilon \in \widehat{D_{\nu}}^{\times} \} \quad \text{for all} \quad {\nu} \in [1,n] \,,
\]
we have
\[
\max \Delta (H) = \max \big(\{d_{\nu}+d_{\nu'} \mid \nu, \nu' \in [1,n] \text{ with }  \nu \neq \nu' \}\cup\{|d_{\nu}|\mid \nu \in [1,n]\}\big) \quad \text{where}
\]

\[
d_{\nu}=\left\{
\begin{aligned}
2&, \qquad&& \text{ if ${G_{\nu}}=\{0\}$ and $s_{\nu}=|\{i\in [1,s_{\nu}]\mid [q_{\nu,i}]=e\}|=2$\,, }\\
0&,&& \text{ if $G_{\nu}=\{0\}$ and  $s_{\nu}=1$; note that $[q_{\nu,1}]=e$ or $|\widehat{D_{\nu}}^{\times}|>1$, since $D_{\nu}$ is not factorial, }\\
-1&,&& \text{ if $G_{\nu}=\{0\}$, $s_{\nu}\ge 2$, and $|\{i\in [1,s_{\nu}]\mid [q_{\nu,i}]=e\}|=0$\,, }\\
1&, && \text{ if $G_{\nu}\neq\{0\}$, or \big($s_{\nu}\ge 2$ and $|\{i\in [1,s_{\nu}]\mid [q_{\nu,i}]=e\}|=1$\big),}\\
&&& \qquad\qquad\qquad\text{ or \big($s_{\nu}\ge 3$ and $|\{i\in [1,s_{\nu}]\mid [q_{\nu,i}]=e\}|\ge 2$\big)}\,.
\end{aligned}\right.
\]
\end{theorem}

\medskip
\medskip
\begin{proof}[Proof of Theorem \ref{1.1} and of Theorem \ref{3.3}]
Let $H$ be a $v$-noetherian weakly Krull monoid as in the formulation of  Theorem \ref{1.1} and of Theorem \ref{3.3}. We proceed in five steps. First, we show that it is sufficient to consider a special class of weakly Krull monoids. Second, we handle the special cases where the $v$-class group is either trivial or infinite, which settles the first statement of Theorem \ref{1.1}. In the third step we prove the second statement of Theorem \ref{1.1}, and in the fourth step we show that $\Delta (H)=\emptyset$ or $\min \Delta (H)=1$. Finally we prove Theorem \ref{3.3}. We use all the notation introduced at the beginning of this section.

\medskip
\noindent
{\bf  1. Reduction to a special case.}  Let $\mathcal H = \{aH \mid a \in H\}$ be the monoid of principal ideals,
$\mathcal I_v^* (H)$ be the monoid of $v$-invertible $v$-ideals of $H$, $\delta_H \colon H_{\red} \to \mathcal I_v^* (H)$ be the canonical monomorphism satisfying $\delta_H (H_{\red}) = \mathcal H$, and $\mathcal C_v (H) = \mathsf q \big( \mathcal I_v^* (H) \big)/\mathsf q ( \mathcal H )$ be  the $v$-class group. We
set $\mathfrak f = (H \DP \widehat H)$, $\mathcal P^* = \{ \mathfrak p \in \mathfrak X(H) \mid \mathfrak p \supset \mathfrak f \}$, and $\mathcal P = \mathfrak X (H) \setminus \mathcal P^*$. By assumption, we have $\mathfrak f \ne \emptyset$. If $\mathfrak f = H$, then $H = \widehat H$ is Krull, and all statements of Theorem \ref{1.1} hold by Proposition \ref{2.3}. Thus we suppose that $\mathfrak f \subsetneq H$ whence $\mathcal P^*$ is finite and non-empty, say $\mathcal P^*  = \{\mathfrak p_1, \ldots, \mathfrak p_n\}$ with $n \in \N$.
By \cite[Theorem 5.5]{Ge-Ka-Re15a}, there exists an isomorphism
\[
\chi \colon \mathcal I_v^*(H) \to D = \mathcal F(\mathcal P) \time (H_{\mathfrak p_1})_{\red} \time \ldots \time (H_{\mathfrak p_n})_{\red}
\]
where $\chi \t \mathcal P = \id_{\mathcal P}$ and, for all $i \in [1,n]$, $D_i := (H_{\mathfrak p_i})_{\red}$ is a reduced seminormal finitely primary monoid, say of rank $s_i$, which is not factorial.
Hence \,$\chi \circ \delta_H \colon H_\red \to \mathcal H = \{aH
\mid a \in H \} \hookrightarrow \mathcal I_v^*(H) \to D$\, induces an isomorphism \,$H_\red \to H^*$, where $H^* \subset D$ is a  cofinal saturated submonoid, and there is a natural isomorphism \,$\overline \chi \colon \mathcal C_v(H) \to \mathsf q (D)/\mathsf q (H^*) = G$ mapping classes of primes onto classes of primes (use \cite[Lemma 4.1]{Ge-Ka-Re15a}). Thus we may assume from now on that $H = H^* \subset D$ is a cofinal saturated submonoid with class group $G = \mathsf q (D)/\mathsf q (H)$.

By Proposition \ref{2.2}, it is sufficient to prove the assertion for the associated $T$-block monoid
\[
B = \mathcal B (G, T, \iota) \subset F = \mathcal F (G) \time T \,,
\]
where $T = D_1 \times \ldots \times D_n$ and $\iota \colon T \to G$ is defined by $\iota (t) = [t]$ for all $t \in T$ (note that, again by Proposition \ref{2.2}, $B$ is seminormal $v$-noetherian weakly Krull with non-trivial conductor and class group isomorphic to $G$). Since $\mathcal B (G) \subset B$ is a divisor-closed submonoid, it follows that  $\Delta  (G)  \subset \Delta (B)$, and in case $|G|\ge 3$, Proposition \ref{2.3} implies that $\Delta (G)$ is an interval with $\min \Delta (G)=1$.

\medskip
\noindent
{\bf  2. Proof of Theorem \ref{1.1}.1.}

Suppose that $|G|=1$. Then $B=F$ and $\mathsf c (B) = \mathsf c (F) = \max \{ \mathsf c (D_1), \ldots, \mathsf c (D_n\}$. Thus Lemma \ref{2.1} implies that $\mathsf c (B) \le 3$ and hence $\Delta (B) \subset \{1\}$.

Suppose that $G$ is infinite. Then $\Delta (G) = \N$ by \cite[Theorem 7.4.1]{Ge-HK06a} and hence $\N \subset \Delta (G)  \subset \Delta (B) \subset \N$.

\medskip
\noindent
{\bf  3. Proof of Theorem \ref{1.1}.2.}

Suppose that $G$ is finite with $|G|>1$, and that either $|\{i \in [1,n] \mid s_i > 1 \}| \le 1$ or that $G$ is an elementary $2$-group.
Lemma \ref{2.1} implies that $\mathsf c (F) = \max \{ \mathsf c (D_1), \ldots , \mathsf c (D_n) \} \le 3$.  Since $B \subset F$ is cofinal saturated with finite class group, the finiteness of $\mathsf c (F)$ implies that $\mathsf c (B) < \infty$
 by \cite[Theorems 3.6.4 and 3.6.7]{Ge-HK06a}. Therefore $\Delta (B)$ is finite and it is sufficient to show that $[1, \max \Delta (B)] \subset \Delta (B)$. We set $m = \max \Delta (B)$, and we use the convention that $\max \Delta (G)=0$ if $\Delta (G)= \emptyset$ (which is the case for $|G|=2$).

We assert that  for each $d \in [\max \Delta (G)+1, m]$, there are $A_d, \ldots, A_m \in B$ such that $\norm{A_d} < \ldots < \norm{A_m}$,  $\max \Delta \big( \mathsf L (A_i) \big) = i$ for each $i \in [d, m]$, and that $\norm{A_i}$ is minimal among all $\norm{A_i'}$ with $A_i' \in B$ and $\max \Delta \big( \mathsf L (A_i') \big) = i$.
This implies that $[d, m] \subset \Delta (B)$ for each $d \in [\max \Delta (G)+1, m]$, and hence
\[
\Delta (B) = \Delta (G) \cup [\max \Delta (G)+1, \max \Delta (B)]
\]
is an interval with $\min \Delta (B)=1$.

We proceed by induction on $d$.
Clearly, the assertion holds for $d=m$. Suppose it holds for some $d \in [\max \Delta (G)+2, \max \Delta (B)]$. By Proposition \ref{3.2} there is an $A_{d-1} \in B$ with $\norm{A_{d-1}} < \norm{A_d}$ and $\max \Delta \big( \mathsf L (A_{d-1}) \big) \ge \max \Delta \big( \mathsf L (A_d) \big) - 1$. The minimality of $\norm{A_d}, \ldots, \norm{A_m}$ implies that $\max \Delta \big( \mathsf L (A_{d-1}) \big) = \max \Delta \big( \mathsf L (A_d) \big) - 1 = d-1$, and hence the assertion follows.

\medskip
\noindent
{\bf 4.} Suppose that $\Delta (B)\ne \emptyset$. We have to verify that $\min \Delta (B)=1$. If $G$ is trivial or infinite, then this follows from  {\bf 2}. If $G$ is finite with $|G| \ge 3$, then $1 \in \Delta (G) \subset \Delta (B)$ by Proposition \ref{2.3}. If $|G|=2$, then $G$ is an elementary $2$-group and the assertion follows from {\bf 3}.

\medskip
\noindent
{\bf 5. Proof of Theorem \ref{3.3}.}

Suppose that $|G|=2$, say $G=\{0,e\}$, and, as in the formulation of Theorem \ref{3.3}, we set $G_{\nu} = \{[\epsilon] \mid \epsilon \in \widehat{D_{\nu}}^{\times} \}$ for all $\nu \in [1,n]$. Note that $\max \mathsf L_F (A) \le 2$ for all $A \in \mathcal A (B)$.
By definition of the catenary degree, there  are $A\in B$ with $\mathsf c (A)=\mathsf c (B)$ and  two factorizations
\begin{equation*}
z_1 =U_1 \cdot \ldots \cdot U_k \in \mathsf Z_B (A) \quad \text{ and} \quad  z_2 =V_1\cdot \ldots \cdot V_l \in \mathsf Z_B (A) \,,
\end{equation*}
where $k\le l$ and $U_1, \ldots, U_k, V_1, \ldots V_l \in \mathcal A (B)$ such that   there is no $(\mathsf c (B)-1)$-chain between $z_1$ and $z_2$.  First we  choose an element $A \in B$ such that $\norm{A}$ is minimal with respect to this property,  and then we choose factorizations $z_1,z_2\in \mathsf Z_B(A)$ such that $|z_1|+|z_2|=k+l$ is maximal with the property that there is no $(\mathsf c (B)-1)$-chain between $z_1$ and $z_2$. Since $\mathcal B (G)$ is factorial and $B$ is not factorial, it follows that $\mathsf c (G)=0 < \mathsf c (B)$. Thus $A \notin \mathcal B (G)$, $|\mathsf Z_B (A)|>1$, and hence there exists an $i\in [1,n]$  such that $\mathsf p_i(A)\neq 1$, say $i=1$.

We start with three assertions {\bf A1}, {\bf A2}, {\bf A3}, and then distinguish five cases.

\medskip
\noindent
{\bf  A1.}  \, $k=2$.

\smallskip
\noindent
{\it Proof of {\bf  A1}.} \, Assume to the contrary that $k\ge 3$. After renumbering if necessary, we may suppose $\mathsf p_1(A)=\epsilon q_{1,1}^{k_1}\cdot \ldots \cdot q_{1,r}^{k_r}q_{1,r+1}^{k_{r+1}}\cdot \ldots \cdot q_{1,{s_1}}^{k_{s_1}}$, where $\epsilon \in \widehat{D_1}^{\times}$, $r\in [0,{s_1}]$, $[q_{1,1}]=\ldots [q_{1,r}]=e$, and $[q_{1,r+1}]=\ldots [q_{1,{s_1}}]=0$.

Suppose that there exist an atom $W\in \mathcal A(B)$, $i_0\in [1,k]$, and $j_0\in [1,l]$ such that $W\t \prod_{i\neq i_0}U_i$ and $W\t \prod_{j\neq j_0}V_j$. Let $A=WU_{i_0}X_1\cdot \ldots \cdot X_{m_1}=WV_{j_0}Y_1\cdot \ldots \cdot Y_{m_2}$, where $X_1,\ldots, X_{m_1}, Y_1,\ldots ,Y_{m_2}\in \mathcal A(B)$. By the minimality of $\norm{A}$, we obtain that there are  $(\mathsf c (B)-1)$-chains between $\prod_{i\neq i_0}U_i$ and $WX_1\cdot \ldots \cdot X_{m_1}$, between $U_{i_0}X_1\cdot \ldots \cdot X_{m_1}$ and  $V_{j_0}Y_1\cdot \ldots \cdot Y_{m_2}$, between   $WY_1\cdot \ldots \cdot Y_{m_2}$ and $\prod_{j\neq j_0}V_j$. Then there is an $(\mathsf c (B)-1)$-chain between $z_1$ and $z_2$, a contradiction.

Therefore we only need to find such a $W$ to get a contradiction.

In fact, if $\mathsf p_1(A)$ is an atom of $D_1$, then $A'=A(\mathsf p_1(A))^{-1}[\mathsf p_1(A)]$ has the definining properties of $A$ but $\norm{A'}<\norm{A}$, a contradiction to the minimality of $\norm{A}$. Therefore $\max \mathsf L_F(\mathsf p_1(A))\ge 2$. If there exists an $\epsilon'\in \widehat{D_1}^{\times}$ such that $\epsilon'q_{1,1}\cdot \ldots \cdot q_{1,{s_1}}\in \mathcal A(B)$, then there exist distinct $i_1,i_2\in [1,k]$ and distinct $j_1,j_2\in [1,l]$ such that $\epsilon'q_{1,1}\cdot \ldots \cdot q_{1,{s_1}}\t U_{i_1}U_{i_2}$ and $\epsilon'q_{1,1}\cdot \ldots \cdot q_{1,{s_1}}\t V_{j_1}V_{j_2}$, and the atom $W=\epsilon'q_{1,1}\cdot \ldots \cdot q_{1,{s_1}}$ has the required property.

Thus we may suppose that $G_1=\{0\}$ and $[q_{1,1}\cdot \ldots \cdot q_{1,{s_1}}]=e$ which implies that $r\ge 1$. We distinguish two cases.

\medskip
\noindent
{ CASE 1: } \, $\max \mathsf L_F(\mathsf p_1(A))=2$\,.
\smallskip

After renumbering if necessary we may suppose that $\mathsf p_1(A)\t U_1U_2$ and $\mathsf p_1(A)\t V_1V_2$. If there exists $i\in [1,r]$  such that $q_{1,i}q_{1,1}\cdot \ldots \cdot q_{1,{s_1}}\t \mathsf p_1(A)$, then we choose $W=q_{1,i}q_{1,1} \cdot \ldots \cdot q_{1,{s_1}}$. Otherwise, it follows that  $k_1=\ldots =k_r=2$, hence  $\mathsf p_1(A)\in \mathcal A(B)$, and we choose $W=\mathsf p_1(A)$.

 \medskip
\noindent
{CASE 2: } \, $\max \mathsf L_F(\mathsf p_1(A))\ge 3$\,.
 \smallskip

 Then $\mathsf p_1(A)\neq \mathsf p_1(U_i)$ for any $i\in [1,k]$ because $\mathsf D(G)=2$.
Without loss of generality, we may suppose that  $\mathsf p_1(U_1)\neq 1$ and $\mathsf p_1(U_2)\neq 1$. We assert that $\mathsf L(U_1U_2)=\{2\}$, and assume to the contrary that $U_1U_2=W_1\cdot \ldots \cdot W_{x}$ where $x\ge 3$ and $W_i\in \mathcal A(B)$ for each $i\in [1,x]$. Then there are  $(\mathsf c (B)-1)$-chains between $U_1\cdot \ldots \cdot U_{k-1}$ and $W_1\cdot \ldots \cdot W_xU_3\cdot \ldots \cdot U_{k-1}$ by the minimality of $\norm{A}$ and there is a  $(\mathsf c (B)-1)$-chains  between $W_1\cdot \ldots \cdot W_xU_3\cdot \ldots \cdot U_{k}$ and  $z_2$ by the maximality of $k+l$. It follows that  there is a $(\mathsf c (B)-1)$-chain between $z_1$ and $z_2$, a contradiction. Thus $\mathsf L (U_1U_2)=\{2\}$.

If $U_1\neq \mathsf p_1(U_1)$ and $U_2\neq \mathsf p_1(U_2)$, then $\mathsf p_1(U_1)\mathsf p_1(U_2)\in \mathcal A(B)$ by $\mathsf L(U_1U_2)=\{2\}$.  Let $U_1U_2=W_1W_2$ with $W_1=\mathsf p_1(U_1)\mathsf p_1(U_2)$. Then there is no $(\mathsf c (B)-1)$-chain between $W_1W_2U_3\cdot \ldots \cdot U_k$ and $V_1\cdot \ldots \cdot V_l$.
Thus we  always may suppose that $U_1=\mathsf p_1(U_1)$ and hence there exists $i\in [1,r]$, say $i=1$, such that $\mathsf v_{q_{1,1}}(U_1)\ge 2$ by $G_1=\{0\}$ and $[q_{1,1}\cdot \ldots \cdot q_{1,{s_1}}]=e$. Therefore $W'=q_{1,1}^2q_{1,2}\cdot \ldots \cdot q_{1,{s_1}}\t U_1\mathsf p_1(U_2)\t U_1U_2$.

 With the same reason and without loss of generality, we   always may suppose that $V_1=\mathsf p_1(V_1)$ and $\mathsf p_1(V_2)\neq 1$.
Then there exists $j\in [1,r]$ such that $W''=q_{1,j}q_{1,1}\cdot \ldots \cdot q_{1,{s_1}}\t V_1\mathsf p_1(V_2)\t V_1V_2$. If $j=1$, then $W'=W''\in \mathcal A(B)$ and  we are done by choosing $W=W'$.  Thus we assume that $j\neq 1$.  After renumbering if necessary we may suppose that $\mathsf p_1(U_i)\neq 1$ for each $i\in [1,k_0]$ and $\mathsf p_1(U_i)= 1$  for each $i\in [k_0+1,k]$ where $k_0\in [2,k]$. After renumbering if necessary we may suppose that $\mathsf p_1(V_i)\neq 1$ for each $i\in [1,l_0]$ and $\mathsf p_1(V_i)= 1$  for each $i\in [l_0+1,l]$ where $l_0\in [2,l]$.
If there exist distinct $j_1,j_2\in [1,l_0]$ such that $W'\t V_{j_1}V_{j_2}$, then we are done by choosing $W=W'$. Otherwise $l_0\ge k_0$. Then there must exist distinct $i_1,i_2\in [1,k_0]$ such that $W''\t U_{i_1}U_{i_2}$ and we are done by choosing $W=W''$. \qed(Proof of {\bf  A1})

\bigskip
If $\mathsf L(A)=\{2\}$, then $\max \Delta(B)+2  \le \mathsf c(B)= \mathsf c (A) = 2$  whence $\Delta (B) = \emptyset$ and $\max \Delta (B)=0=\mathsf c (B)-2$. Suppose that $\mathsf L(A)\neq\{2\}$, say
\begin{equation}\label{equation1}
A=U_1U_2=W_1\cdot \ldots \cdot W_{m}
\end{equation}
where $m=\min \big(\mathsf L(A)\setminus\{2\}\big)$ and $W_1, \ldots, W_{m} \in \mathcal A(B)$. Since $m+l>2+l$, we obtain that there is an $(\mathsf c (B)-1)$-chain between $z_2$ and $W_1\cdot \ldots \cdot W_{m}$. Therefore there is no $(\mathsf c (B)-1)$-chain between $U_1U_2$ and $W_1\cdot \ldots \cdot W_{m}$. It follows that $\max\Delta(B)+2\ge m\ge \mathsf c(B)\ge \max\Delta(B)+2$ and hence  $\max\Delta(B)=m-2=\mathsf c (B)-2$.  Since $G$ is an elementary $2$-group, Theorem \ref{1.1}.2 implies that  $\Delta(B)$ is an interval.

We set $m=2$ in case $\mathsf L (A)=\{2\}$, and then it remains to prove that
\begin{equation}\label{equation2}
m-2 = \max \big(\{d_{\nu}+d_{\nu'} \mid \nu, \nu' \in [1,n] \text{ with }  \nu \neq \nu' \}\cup\{|d_{\nu}|\mid \nu \in [1,n]\}\big)\,.
\end{equation}

For every  $j\in [1,n]$ we may suppose,  after renumbering if necessary,  that $[q_{j,i}]=e$ for each $i\in [1,s_j']$ and $[q_{j,i}]=0$ for each $i\in [s_j'+1,s_j]$ where $s_j'\in [0,s_j]$.

\medskip
\noindent
{\bf A2.} Let $j \in [1,n]$.
\begin{enumerate}
\item If $s_j=1$, $a,b\in D_j\cap\mathcal A( B)$, then there exists $\epsilon\in \widehat{D_j}^{\times}$ such that $b=\epsilon a$.

\item  If $s_j\ge 2$ and $a\in D_j\cap B$, then
 \begin{align*}
\min \mathsf L(a)\le \left\{
\begin{aligned}
&3,\quad  &&\text{ if $s_j=s_j'=2$ and $G_j=\{0\}$\ i.e. $d_j=2$\,,}\\
&2, &&\text{ otherwise.}\\
\end{aligned}\right.
\end{align*}
\end{enumerate}

\smallskip
\noindent
{\it Proof of {\bf A2}. } \,
We prove the assertion for $j=1$ and set $s'= s_1'$.

\smallskip
1. If $[q_{1,1}]=0$, then $a=\epsilon_1 q_{1,1}$, $b=\epsilon_2 q_{1,1}$, where $\epsilon_1,\epsilon_2\in \widehat{D_1}^{\times}$ and hence the assertion follows. Thus we assume that $[q_{1,1}]=e$.  If $G_1=\{0\}$, then $a=\epsilon_1 q_{1,1}^2$, $b=\epsilon_2 q_{1,1}^2$, where $[\epsilon_1]=[\epsilon_2]=0$ and hence the assertion follows.   If $G_1\neq \{0\}$ , then $a=\epsilon_1 q_{1,1}$, $b=\epsilon_2 q_{1,1}$, where $[\epsilon_1]=[\epsilon_2]=e$ and hence the assertion follows.
\smallskip

2. Suppose $s'=2$ and $G_1=\{0\}$. If $s_1=2$, we assume to the contrary that $\min \mathsf L(a)\ge 4$ and  let $a=\epsilon q_{1,1}^{k_1} q_{1,2}^{k_{2}}$ with $k_i\ge 4$ for each $i\in [1,2]$. Then $k_1+k_2$ is even. If $k_1$ is even, then $\epsilon q_{1,1}q_{1,2}^{k_2-1}$ and $q_{1,1}^{k_1-1}q_{1,2}$ are two atoms of $B$. Thus $\min \mathsf L(a)\le 2$, a contradiction.  If $k_1$ is odd, then $\epsilon q_{1,1}q_{1,2}$, $q_{1,1}q_{1,2}^{k_2-2}$ and $q_{1,1}^{k_1-2}q_{1,2}$ are three atoms of $B$. Thus $\min \mathsf L(a)\le 3$, a contradiction.
If $s_1\ge 3$, we let $a=\epsilon q_{1,1}^{k_1}\cdot \ldots \cdot q_{1,s_1}^{k_{s_1}}$ with $k_i\ge 2$ for each $i\in [1,s_1]$. Then $k_1+k_2$ is even. Since $\epsilon q_{1,1}q_{1,2}q_{1,3}^{k_3-1}\cdot \ldots \cdot q_{1,s_1}^{k_{s_1}-1}$ and $q_{1,1}^{k_1-1}q_{1,2}^{k_2-1}q_{1,3}\cdot \ldots \cdot q_{1,s_1}$ are two atoms of $B$. Thus $\min \mathsf L(a)\le 2$.

For the other cases, we have that $G_1\neq \{0\}$, or $s'\ge 3$, or $s'\le 1$.

If $G_1\neq \{0\}$, then let $a=\epsilon q_{1,1}^{k_1}\cdot \ldots \cdot q_{1,s_1}^{k_{s_1}}$ with $k_i\ge 2$ for each $i\in [1,s_1]$. There exists $\epsilon'\in \widehat{D_1}^{\times}$ such that $\epsilon'q_{1,1}q_{1,2}^{k_2-1}\cdot \ldots \cdot q_{1,s_1}^{k_{s_1}-1}$ and $\epsilon\epsilon'q_{1,1}^{k_1-1}q_{1,2}\cdot \ldots \cdot q_{1,s_1}$ are two atoms of $B$ and hence $\min \mathsf L(a)\le 2$. Then we always assume that $G_1= \{0\}$.

 If $s'\ge 3$, we assume to the contrary that $\min \mathsf L(ab)\ge 3$ and  let $a=\epsilon q_{1,1}^{k_1}\cdot \ldots \cdot q_{1,s_1}^{k_{s_1}}$ with $k_i\ge 3$ for each $i\in [1,s_1]$. We choose $\delta\in [1,2]$ such that $q_{1,1}q_{1,2}^{k_2-1}q_{1,3}^{\delta}q_{1,4}^{k_4-1}\cdot \ldots \cdot q_{1,s_1}^{k_{s_1}-1}$ is an atom of $B$ and hence $\epsilon q_{1,1}^{k_1-1}q_{1,2}q_{1,3}^{k_3-\delta}q_{1,4}\cdot \ldots \cdot q_{1,s_1}$ is also an atom of $B$. Therefore $\min\mathsf L(a)\le 2$, a contradiction.

  Now we assume that $s'\le 1$ (note $s_1\ge 2$) and $a=\epsilon q_{1,1}^{k_1}\cdot \ldots \cdot q_{1,s_1}^{k_{s_1}}$ with $k_i\ge 2$ for each $i\in [1,s_1]$. Suppose that $s'=1$. Then $k_1$ is even. If $k_1=2$, then $a$ is an atom of $B$. Otherwise $k_1\ge 4$. It follows that $ q_{1,1}^{2}q_{1,2}^{k_2-1}\cdot \ldots \cdot q_{1,s_1}^{k_{s_1}-1}$ and $\epsilon q_{1,1}^{k_1-2}q_{1,2}q_{1,3}\cdot \ldots \cdot q_{1,s_1}$ are two atoms of $B$. Hence $\min \mathsf L(a)\le 2$. Suppose that $s'=0$. Then $\epsilon q_{1,1}q_{1,2}^{k_2-1}\cdot \ldots \cdot q_{1,s_1}^{k_{s_1}-1}$ and $q_{1,1}^{k_1-1}q_{1,2}q_{1,3}\cdot \ldots \cdot q_{1,s_1}$ are two atoms of $B$. Thus $\min \mathsf L(a)\le 2$. \qed(Proof of {\bf A2})

\medskip
\noindent
{\bf A3.} Let $a,b\in \mathcal A(B)$.
\begin{enumerate}
\item Suppose that  there exists $i\in [1,n]$ such that $a\in D_i$ and $\mathsf p_i(b)\neq 1$. If $s_i=1$, then $\mathsf L(ab)=\{2\}$. If  $\mathsf L(ab)\neq \{2\}$, then \begin{align*}
    \min \big(\mathsf L(ab)\setminus\{2\} \big) \le \left\{
    \begin{aligned}
    &4, \text{ if $d_i=2$,}\\
    &3,  \text{ if $d_i\neq 2$.}
    \end{aligned}\right.
    \end{align*}

\item Suppose that $a=a_1a_2$ and $b=b_1b_2$ with $a_1, a_2, b_1, b_2 $ are atoms of $F$, that there exists  $i\in [1,n]$ such that $\mathsf p_i(a)=a_1$, $\mathsf p_i(b)=b_1$, and that $a_2=e$ or \big($a_2\in D_j$ and $b_2\not\in D_j$ where $j\in [1,n]\setminus\{i\}$\big).   Then $d_i\neq -1$. If $d_i=0$, then $\mathsf L(ab)=\{2\}$. If $\mathsf L(ab)\neq\{2\}$, then \begin{align*}
    \min \big(\mathsf L(ab)\setminus\{2\} \big) \le \left\{
    \begin{aligned}
    &4, \text{ if $d_i=2$,}\\
    &3,  \text{ if $d_i=1$.}
    \end{aligned}\right.
     \end{align*}

\item Suppose that $a=a_1a_2$ and $b=b_1b_2$ with $a_1, a_2,b_1, b_2 $ are atoms of $F$ and there exist distinct $i,j\in [1,n]$ such that $\mathsf p_i(a)=a_1$, $\mathsf p_j(a)=a_2$, $\mathsf p_i(b)=b_1$, and $\mathsf p_j(b)=b_2$. Then $d_i\neq -1$ and $d_j\neq -1$.  If $d_i+d_j=0$, then $\mathsf L(ab)=\{2\}$. If $\mathsf L(ab)\neq\{2\}$, then
    \begin{align*}
         \min \big(\mathsf L(ab)\setminus\{2\} \big) \le \left\{
         \begin{aligned}
         &6, \text{ if $d_i+d_j=4$,}\\
         &5,  \text{ if $d_i+d_j=3$,}\\
         &4,   \text{ if $d_i+d_j=2$,}\\
         &3,    \text{ if $d_i+d_j=1$.}\\
         \end{aligned}\right.
    \end{align*}
\end{enumerate}

\smallskip
\noindent
{\it Proof of {\bf A3.} } 1.
If $s_i=1$, we assume to the contrary that  $ab=x_1\cdot \ldots \cdot x_{\mu}$ with $\mu\ge 3$ and $x_k\in \mathcal A(B)$ for each $k\in [1,\mu]$. Then there exists $k\in [1,\mu]$, say $k=1$, such that $x_1\in D_i$ and hence $x_1=\epsilon a$ where $[\epsilon]=0$ by {\bf A2}.1. Therefore $\epsilon^{-1} b=x_2 \cdot \ldots \cdot x_{\mu}$ is also an atom of $B$, a contradiction.

 Suppose that $\mathsf L(ab)\neq\{2\}$ and $ab=x_1\cdot \ldots \cdot x_{\mu}$ with $\mu\ge 3$ and $x_k\in \mathcal A(B)$ for each $k\in [1,\mu]$. Let $b=b_1b_2$ with $b_1=\mathsf p_i(b)$.   Then there must exist $k\in [1,\mu]$, say $k=1$, such that $b_2\t x_1$. Therefore $abx_1^{-1}\in D_i$ and $abx_1^{-1}\not\in \mathcal A(B)$. Then  by {\bf A2}
$$\min \big(\mathsf L(ab)\setminus\{2\}\big) \le 1+\min \mathsf L(abx_1^{-1})\le\left\{\begin{aligned}
&4, \text{if $d_i=2$,}\\
&3, \text{ otherwise.}
\end{aligned}\right. $$
     \smallskip

2. By definition of $d_i$ and the existence of $a$, we have that $d_i\neq -1$. If $d_i=0$, then $\mathsf L(ab)=\{2\}$ is  obvious by definition. Suppose $\mathsf L(ab)\neq \{2\}$ and  $ab=a_1b_1\cdot a_2b_2=x_1\cdot \ldots \cdot x_{\mu}$ with $\mu\ge 3$ and $x_k\in \mathcal A(B)$ for each $k\in [1,\mu]$. Then there exist distinct $k,j\in [1,\mu]$, say $k=2, j=3$, such that $a_2\t x_2$ and $b_2\t x_3$ which implies that $x_1\t a_1b_1$. By our assumption,  $a_2b_2\in \mathcal A(B)$ and hence $a_1b_1\neq x_1$. Therefore $\min \big(\mathsf L(ab)\setminus\{2\}\big) \le \min \mathsf L(a_1b_1)+\min \mathsf L(a_2b_2)\le d_i+2$ by {\bf A2}.

\smallskip
3.  By definition of $d_i$ and the existence of $a$, we have that $d_i\neq -1$ and $d_j\neq -1$. If $d_i+d_j=0$, then $d_i=d_j=0$ and hence $\mathsf L(ab)=\{2\}$. Suppose that $\mathsf L(ab)\neq \{2\}$ and  $ab=x_1\cdot \ldots \cdot x_{\mu}$ with $\mu\ge 3$ and $x_k\in \mathcal A(B)$ for each $k\in [1,\mu]$.  If $a_1b_1, a_2b_2\in \mathcal A(B)$, then $x_k\nmid a_1b_1$ and $x_k\nmid a_2b_2$ for each $k\in [1,\mu]$.  Thus $a_1b_1=\mathsf p_i(x_1)\ldots \mathsf p_i(x_{\mu})$ with $\mathsf p_i(x_{\nu})\neq 1$ for each $\nu\in [1,\mu]$, a contradiction to $\mathsf D(G)=2$. Therefore $a_1b_1\not\in \mathcal A(B)$ or $a_2b_2\not\in \mathcal A(B)$. It follows that  $\min \big(\mathsf L(ab)\setminus\{2\} \big) \le \min \mathsf L(a_1b_1)+\min \mathsf L(a_2b_2)\le d_i+d_j+2$ by {\bf A2}.
 \qed(Proof of {\bf A3})

\medskip
Note, if  ($\mathsf p_i(U_1)=1$ or $\mathsf p_i(U_2)=1$) for every $i\in [1,n]$, then $\mathsf L(A)=\{2\}$. Therefore  $\mathsf L(A)\neq\{2\}$ implies that there exists $i_0\in [1,n]$ such that $\mathsf p_{i_0}(U_1)\neq 1$ and $\mathsf p_{i_0}(U_2)\neq 1$.

Now we distinguish five cases depending on the size of the right hand side of Equation \ref{equation2}.

\medskip
\noindent
{CASE 1 :} \, $\max \big(\{d_{\nu}+d_{\nu'} \mid \nu, \nu' \in [1,n] \text{ with }  \nu \neq \nu' \}\cup\{|d_{\nu}|\mid \nu \in [1,n]\}\big)=4$.

Then there exist distinct $\nu, \nu' \in [1,n]$ such that $d_{\nu}=d_{\nu'}=2$, say $d_1=d_2=2$. We define $U_1'=q_{1,1}q_{1,2}^2 q_{2,1}q_{2,2}^2$ and $U_2'=q_{1,1}^2q_{1,2} q_{2,1}^2q_{2,2}$. Then $U_1',U_2'\in \mathcal A(B)$ and $\mathsf L(U_1'U_2')=\{2,6\}$ which implies that $4\in \Delta(B)$.
By {\bf A3}, we know that $\min \Delta \big( \mathsf L (U_1U_2) \setminus \{2\} \big) \le 6$. Thus $\max \Delta(B)\le 4$  and hence $\max \Delta(B)=4$.

\medskip
\noindent
{CASE 2 :} \, $\max \big(\{d_{\nu}+d_{\nu'} \mid \nu, \nu' \in [1,n] \text{ with }  \nu \neq \nu' \}\cup\{|d_{\nu}|\mid \nu \in [1,n]\}\big)=3$.

Then there exist distinct $\nu, \nu' \in [1,n]$ such that $d_{\nu}=2$, $d_{\nu'}=1$, and $d_{\lambda} \le 1$ for each $\lambda \in [1,n]\setminus\{\nu, \nu'\}$, say $d_1=2$ and $d_2=1$.  Since $d_2=1$, we set
   \begin{align*}
  & a_1=\left\{\begin{aligned}
   &\epsilon q_{2,1}\cdot \ldots \cdot q_{2,s_2}&&\text{ with }[\epsilon]+[q_{2,1}\cdot \ldots \cdot q_{2,s_2}]=e,   \text{ if $G_2\neq \{0\}$}\\
   &q_{2,1}q_{2,2}\cdot \ldots \cdot q_{2,s_2},             &&\text{ if $G_2=\{0\}$, $s_2\ge 2$, and $[q_{2,1}]=e$, $[q_{2,j}]=0$ for each $j\in [2,s_2]$}\\
   &q_{2,1}q_{2,2}^{\delta}q_{2,3}\cdot \ldots \cdot q_{2,s_2}, &&\text{ with }\delta\in[1,2]\text{ such that }[q_{2,2}^{\delta}]=[q_{2,3}\cdot \ldots \cdot q_{2,s_2}], \\
   &&&  \qquad\qquad \qquad\qquad        \text{ if $G_2=\{0\}$, $s_2\ge 3$, and   $[q_{2,1}]=[q_{2,2}]=e$}\\
   \end{aligned}\right.\\
  & a_2=\left\{\begin{aligned}
      &a_1,  &&\text{ if $G_2\neq \{0\}$}\\
      &q_{2,1}^3q_{2,2}\cdot \ldots \cdot q_{2,s_2},            && \text{ if  $G_2=\{0\}$, $s_2\ge 2$, and $[q_{2,1}]=e$, $[q_{2,j}]=0$ for each $j\in [2,s_2]$}\\
        &q_{2,1}^3q_{2,2}^{\delta}q_{2,3}\cdot \ldots \cdot q_{2,s_2},&& \text{ with }\delta\in[1,2]\text{ such that }[q_{2,2}^{\delta}]=[q_{2,3}\cdot \ldots \cdot q_{2,s_2}], \\
           &&&  \qquad\qquad \qquad\qquad        \text{ if $G_2=\{0\}$, $s_2\ge 3$, and   $[q_{2,1}]=[q_{2,2}]=e$}\\
      \end{aligned}\right.
   \end{align*}
and define $U_1'=a_1\cdot q_{1,1}q_{1,2}^2$ and $U_2'=a_2\cdot q_{1,1}^2q_{1,2}$. Then $U_1',U_2'\in \mathcal A(B)$ and $\mathsf L(U_1'U_2')=\{2,5\}$ which implies that $3\in \Delta(B)$.
 By {\bf A3}, we know that $\min \Delta \big( \mathsf L (U_1U_2) \setminus \{2\} \big) \le 5$. Thus $\max \Delta(B)\le 3$  and hence $\max \Delta(B)=3$.

\medskip
\noindent {CASE 3 :} \, $\max \big(\{d_{\nu}+d_{\nu'} \mid \nu, \nu' \in [1,n] \text{ with }  \nu \neq \nu' \}\cup\{|d_{\nu}|\mid \nu \in [1,n]\}\big)=2$.

Then there exist distinct $\nu, \nu' \in [1,n]$ such that $d_{\nu}=d_{\nu'}=1$,
     and $d_{\lambda} \le 1$ for each $\lambda \in [1,n]\setminus\{\nu, \nu' \}$, or there exists  $\nu \in [1,n]$ such that $d_{\nu} =2$
            and $d_{\lambda} \le 0$ for each $\lambda \in [1,n] \setminus\{\nu \}$\,.

We start with the first case and, after renumbering if necessary, we suppose that $d_1=d_2=1$. We  set
      \begin{align*}
        & a_1=\left\{\begin{aligned}
         &\epsilon q_{1,1}\cdot \ldots \cdot q_{1,s_1}&&\text{ with }[\epsilon]+[q_{1,1}\cdot \ldots \cdot q_{1,s_1}]=e,   \text{ if $G_1\neq \{0\}$}\\
         &q_{1,1}q_{1,2}\cdot \ldots \cdot q_{1,s_1},             &&\text{ if $G_1=\{0\}$, $s_1\ge 1$, and $[q_{1,1}]=e$, $[q_{1,j}]=0$ for each $j\in [1,s_1]$}\\
         &q_{1,1}q_{1,2}^{\delta}q_{1,3}\cdot \ldots \cdot q_{1,s_1}, &&\text{ with }\delta\in[1,2]\text{ such that }[q_{1,2}^{\delta}]=[q_{1,3}\cdot \ldots \cdot q_{1,s_1}], \\
         &&&  \qquad\qquad \qquad\qquad        \text{ if $G_1=\{0\}$, $s_1\ge 3$, and   $[q_{1,1}]=[q_{1,2}]=e$}\\
         \end{aligned}\right.\\
        & a_2=\left\{\begin{aligned}
            &a_1,  &&\text{ if $G_1\neq \{0\}$}\\
            &q_{1,1}^3q_{1,2}\cdot \ldots \cdot q_{1,s_1},            && \text{ if  $G_1=\{0\}$, $s_1\ge 1$, and $[q_{1,1}]=e$, $[q_{1,j}]=0$ for each $j\in [1,s_1]$}\\
              &q_{1,1}^3q_{1,2}^{\delta}q_{1,3}\cdot \ldots \cdot q_{1,s_1},&& \text{ with }\delta\in[1,1]\text{ such that }[q_{1,2}^{\delta}]=[q_{1,3}\cdot \ldots \cdot q_{1,s_1}], \\
                 &&&  \qquad\qquad \qquad\qquad        \text{ if $G_1=\{0\}$, $s_1\ge 3$, and   $[q_{1,1}]=[q_{1,2}]=e$}\\
            \end{aligned}\right.\\
            \end{align*}
            \begin{align*}
             & b_1=\left\{\begin{aligned}
                     &\epsilon q_{2,1}\cdot \ldots \cdot q_{2,s_2}&&\text{ with }[\epsilon]+[q_{2,1}\cdot \ldots \cdot q_{2,s_2}]=e,   \text{ if $G_2\neq \{0\}$}\\
                     &q_{2,1}q_{2,2}\cdot \ldots \cdot q_{2,s_2},             &&\text{ if $G_2=\{0\}$, $s_2\ge 2$, and $[q_{2,1}]=e$, $[q_{2,j}]=0$ for each $j\in [2,s_2]$}\\
                     &q_{2,1}q_{2,2}^{\delta}q_{2,3}\cdot \ldots \cdot q_{2,s_2}, &&\text{ with }\delta\in[1,2]\text{ such that }[q_{2,2}^{\delta}]=[q_{2,3}\cdot \ldots \cdot q_{2,s_2}], \\
                     &&&  \qquad\qquad \qquad\qquad        \text{ if $G_2=\{0\}$, $s_2\ge 3$, and   $[q_{2,1}]=[q_{2,2}]=e$}\\
                     \end{aligned}\right.\\
                    & b_2=\left\{\begin{aligned}
                        &a_1,  &&\text{ if $G_2\neq \{0\}$}\\
                        &q_{2,1}^3q_{2,2}\cdot \ldots \cdot q_{2,s_2},            && \text{ if  $G_2=\{0\}$, $s_2\ge 2$, and $[q_{2,1}]=e$, $[q_{2,j}]=0$ for each $j\in [2,s_2]$}\\
                          &q_{2,1}^3q_{2,2}^{\delta}q_{2,3}\cdot \ldots \cdot q_{2,s_2},&& \text{ with }\delta\in[1,2]\text{ such that }[q_{2,2}^{\delta}]=[q_{2,3}\cdot \ldots \cdot q_{2,s_2}], \\
                             &&&  \qquad\qquad \qquad\qquad        \text{ if $G_2=\{0\}$, $s_2\ge 3$, and   $[q_{2,1}]=[q_{2,2}]=e$}\\
                        \end{aligned}\right.
         \end{align*}
and define       $U_1'=a_1b_1$ and $U_2'=a_2b_2$. Then $U_1',U_2' \in \mathcal A (B)$ and $\mathsf L(U_1'U_2')=\{2,4\}$ which implies that $2\in \Delta(B)$.
   \medskip

Now we consider the second case and suppose that there exists  $\nu \in [1,n]$ such that $d_{\nu} =2$ and $d_{\lambda} \le 0$ for each $\lambda \in [1,n]\setminus\{\nu \}$, say $\nu=1$.
We define $U_1'=e\cdot q_{1,1}q_{1,2}^2$ and $U_2'=e\cdot q_{1,1}^2q_{1,2}$. Then $U_1',U_2'\in \mathcal A(B)$ and $\mathsf L(U_1'U_2')=\{2,4\}$ which implies that $2\in \Delta(B)$.

Therefore in both cases, we have that $2\in \Delta(B)$.
 By {\bf A3}, we know that $\min \Delta \big( \mathsf L (U_1U_2) \setminus \{2\} \big) \le 4$. Thus $\max \Delta(B)\le 2$  and hence $\max \Delta(B)=2$.

\medskip
\noindent {CASE 4 : } \, $\max \big(\{d_{\nu}+d_{\nu'} \mid \nu, \nu' \in [1,n] \text{ with }  \nu \neq \nu' \}\cup\{|d_{\nu}|\mid \nu \in [1,n]\}\big)=1$.

Then there exists $\nu \in [1,n]$ such that $d_{\nu} =1$  and $d_{\lambda} \le 0$ for each $\lambda \in [1,n]\setminus\{\nu \}$, or there exists  $\nu \in [1,n]$ such that $d_{\nu} =-1$ and $d_{\lambda} =0$ for each $\lambda \in [1,n]\setminus\{\nu \}$\,.

We start with the first case and, after renumbering if necessary, we    suppose that $d_1=1$.  We set
          \begin{align*}
                  & a_1=\left\{\begin{aligned}
                   &\epsilon q_{1,1}\cdot \ldots \cdot q_{1,s_1}&&\text{ with }[\epsilon]+[q_{1,1}\cdot \ldots \cdot q_{1,s_1}]=e,   \text{ if $G_1\neq \{0\}$}\\
                   &q_{1,1}q_{1,2}\cdot \ldots \cdot q_{1,s_1},             &&\text{ if $G_1=\{0\}$, $s_1\ge 1$, and $[q_{1,1}]=e$, $[q_{1,j}]=0$ for each $j\in [1,s_1]$}\\
                   &q_{1,1}q_{1,2}^{\delta}q_{1,3}\cdot \ldots \cdot q_{1,s_1}, &&\text{ with }\delta\in[1,2]\text{ such that }[q_{1,2}^{\delta}]=[q_{1,3}\cdot \ldots \cdot q_{1,s_1}], \\
                   &&&  \qquad\qquad \qquad\qquad        \text{ if $G_1=\{0\}$, $s_1\ge 3$, and   $[q_{1,1}]=[q_{1,2}]=e$}\\
                   \end{aligned}\right.\\
                  & a_2=\left\{\begin{aligned}
                      &a_1,  &&\text{ if $G_1\neq \{0\}$}\\
                      &q_{1,1}^3q_{1,2}\cdot \ldots \cdot q_{1,s_1},            && \text{ if  $G_1=\{0\}$, $s_1\ge 1$, and $[q_{1,1}]=e$, $[q_{1,j}]=0$ for each $j\in [1,s_1]$}\\
                        &q_{1,1}^3q_{1,2}^{\delta}q_{1,3}\cdot \ldots \cdot q_{1,s_1},&& \text{ with }\delta\in[1,1]\text{ such that }[q_{1,2}^{\delta}]=[q_{1,3}\cdot \ldots \cdot q_{1,s_1}], \\
                           &&&  \qquad\qquad \qquad\qquad        \text{ if $G_1=\{0\}$, $s_1\ge 3$, and   $[q_{1,1}]=[q_{1,2}]=e$}\\
                      \end{aligned}\right.\\
                   \end{align*}
and define $U_1'=e\cdot a_1$ and $U_2'=e\cdot a_2$. Then $U_1', U_2' \in \mathcal A (B)$ and $\mathsf L(U_1'U_2')=\{2,3\}$ which implies that $1\in \Delta(B)$.
   \medskip

Now we consider the second case and suppose there exists  $\nu \in [1,n]$ such that $d_{\nu} =-1$ and $d_{\lambda}=0$ for each $\lambda \in [1,n]\setminus\{\nu\}$, say $\nu=1$.
We define $U_1'=q_{1,1}^2q_{1,2}\cdot \ldots \cdot q_{1,s_1}$ and $U_2'=q_{1,1}q_{1,2}^2\cdot \ldots \cdot q_{1,s_1}^2$. Then $U_1', U_2' \in \mathcal A (B)$ and  $\mathsf L(U_1'U_2')=\{2,3\}$ which implies that $1\in \Delta(B)$.

Therefore in both cases, we have that $1\in \Delta(B)$.  By {\bf A3}, we know that $\min \Delta \big( \mathsf L (U_1U_2) \setminus \{2\} \big) \le 3$. Thus $\max \Delta(B)\le 1$  and hence $\max \Delta(B)=1$.

\medskip
\noindent
{CASE 5 :} \, $\max \big(\{d_{\nu}+d_{\nu'} \mid \nu, \nu' \in [1,n] \text{ with }  \nu \neq \nu' \}\cup\{|d_{\nu}|\mid \nu \in [1,n]\}\big)=0$.

Then $d_{\nu}=0$ for each $\nu \in[1,n]$.  We have $U_1=\mathsf p_{i_0}(U_1)=\epsilon p_{i_0,1}^{\delta}=\epsilon'U_2=\epsilon' \mathsf p_{i_0}(U_2)$, where $\delta\in [1,2]$ and $\epsilon,\epsilon'\in \widehat{D_i}^{\times}$, which implies that $\mathsf L(U_1U_2)=\{2\}$, a contradiction.
\end{proof}

\medskip
We provide a list of $v$-noetherian weakly Krull monoids having nontrivial conductor and finite $v$-class group. However, they are either not seminormal or they miss the assumption on the prime ideals in the classes, and the statements of Theorem \ref{1.1} fail (i.e.,  $\min \Delta (H)>1$ or $\Delta (H)$ is not an interval).

\medskip
\begin{examples} \label{3.4}~

\medskip
1. (Krull monoids) By definition, every Krull monoid is a seminormal $v$-noetherian weakly Krull monoid. Let $H$ be a Krull monoid with class group $G$ and let $G_P \subset G$ denote the set of classes containing minimal prime ideals. If $G_P=G$, then $\Delta (H)$ is an interval (Proposition \ref{2.3}). Suppose that $G_P \ne G$. Then, in general, the set of distances $\Delta (H)$ need not be an interval. There is an abundance of natural examples, and all these phenomena already occur in Dedekind domains (see \cite[Remark 3.1]{Ge-Yu12b} and \cite[Theorem 3.7.8]{Ge-HK06a}).

\medskip
2. (Weakly factorial monoids) A monoid is  weakly factorial if every non-unit is a finite product of primary elements (equivalently, if it is weakly Krull with trivial $t$-class group, see \cite[Exercise 22.5]{HK98}). In particular, primary monoids are weakly factorial.  To recall the connection between ring theoretical and monoid theoretical concepts, let $R$ be a domain. Then its multiplicative monoid $R^{\bullet}$ is primary if and only if $R$ is one-dimensional and local and, if $R$ is a one-dimensional local Mori domain with $(R \colon \widehat R) \ne \{0\}$, then $R^{\bullet}$ is finitely primary; furthermore, $R^{\bullet}$ is seminormal finitely primary  if and only if $R$ is a seminormal one-dimensional local Mori domain (\cite[Proposition 2.10.7]{Ge-HK06a} and \cite[Lemma 3.4]{Ge-Ka-Re15a}).

 The following examples are $v$-noetherian weakly Krull monoids with nontrivial conductor and trivial $v$-class group. However, they fail to be seminormal and their sets of distances are not intervals.

\smallskip
2.(a) (Numerical monoids) Numerical monoids are finitely generated (and hence $v$-noetherian) finitely primary monoids of rank one, and hence they are weakly Krull with nontrivial conductor and trivial $v$-class group. Let $H$ be a numerical monoid. Then, in general, we have $2 + \max \Delta (H) < \mathsf c (H)$ (see, for example,  \cite[Example 3.1.6]{Ge-HK06a}).  Sets of distances of numerical monoids (and in particular, gaps in their sets of distances) have found wide interest in the literature. To mention an explicit example, if $H = \langle n, n+1, n^2-n-1\rangle$ with $n \ge 3$, then $\Delta (H) = [1, n-2] \cup \{2n-5\}$ by \cite[Proposition 4.9]{B-C-K-R06}. Furthermore, each set of the form $\{ d, td\}$ with $d, t \in \N$ occurs as a set of distances of a numerical monoid (\cite{Co-Ka15a}).

\smallskip
2.(b) (Finitely primary monoids of higher rank) For each $d \in \N$ there is a $v$-noetherian finitely primary monoid of rank two  with $\min \Delta (H)=d$ (\cite[Example 3.1.9]{Ge-HK06a}).

\medskip
3. (Seminormal $v$-noetherian weakly Krull monoids with nontrivial conductor) Consider the seminormal $v$-noetherian finitely primary monoid
\[
D = \{p_1^{k_1}p_2^{k_2} \mid k_1, k_2 \in \N \} \cup \{1\} \subset \widehat D = \mathcal F ( \{p_1, p_2\}) \,,
\]
a finite cyclic group $G$ of order $|G|=n\ge 3$, and an element $e \in G$ with $\ord (e)=n$. We define a homomorphism $\iota \colon D \to G$ by setting $\iota (p_1)=e$ and $\iota (p_2)=-e$.

\smallskip
3.(a) Then $H = \Ker (\iota) \hookrightarrow D$ is a cofinal saturated  submonoid (\cite[Proposition 2.5.1]{Ge-HK06a}) and it is a seminormal $v$-noetherian weakly Krull monoid with $(H \colon \widehat H) \ne \emptyset$ by \cite[Lemma 5.2]{Ge-Ka-Re15a}. We assert that $\min \Delta (H)=n$.

Clearly, $H = \{ p_1^{k_1}p_2^{k_2} \mid k_1, k_2 \in \N, k_1 \equiv k_2 \mod n \} \cup \{1\}$ and
\[
\mathcal A (H) = \{p_1 p_2^{k_2} \mid k_2 \in 1+ n\N_0 \} \cup \{ p_1^{k_1} p_2 \mid k_1 \in 1+n\N_0 \} \,.
\]
Thus, if $u_1 \cdot \ldots \cdot u_k =v_1 \cdot \ldots \cdot v_{\ell}$, where $k, \ell \in \N$ and $u_1, \ldots, u_k, v_1, \ldots, v_{\ell} \in \mathcal A (H)$, then $k \equiv \ell \mod n$ and hence $n$ divides $\gcd \Delta (H) = \min \Delta (H)$. To show that $n \in \Delta (H)$, consider the element
$a = p_1^{n+2} p_2^{n+2} \in H$. Clearly,  $\{p_1p_2, p_1p_2^{n+1}, p_1^{n+1}p_2\}$ is the set of atoms of $H$ dividing $a$, $a = (p_1p_2^{n+1})(p_1^{n+1}p_2)= (p_1p_2)^{n+2}$, and hence $\mathsf L_H (a) = \{2, n+2\}$.

\smallskip
3.(b) If the above monoid occurs as the primary component of a $T$-block monoid, then the situation is different. To show this, let us consider the monoid
\[
B = \mathcal B (G, T, \iota) \subset \mathcal F (G) \times D \,,
\]
where $G, \iota$, and $D$ are as at the beginning of 3. Then $B$  satisfies all assumptions of Theorem \ref{1.1}, $\mathcal B (G) \subset B$  and $H \subset B$ are divisor-closed submonoids. We assert that $\Delta (B) = [1,n]$. Since $\Delta  (G)= [1,n-2]$ by Propositions \ref{2.3} and \ref{2.4}, it follows that $[1, n-2] \subset \Delta (B)$.  Furthermore, $a = p_1^{n+2} p_2^{n+2} \in H \subset B$, $\mathsf L_H (a) = \mathsf L_B (a)= \{2, n+2\}$, and hence $n \in \Delta (B)$.  The element $b = p_1^{n+1}p_2^{n+2}e \in B$, $\{p_1p_2, p_1p_2^{n+1}, ep_1p_2^2, ep_1^n p_2\}$ is the set of atoms of $B$ dividing $b$, $b = (p_1p_2^{n+1})(ep_1^np_2)=(p_1p_2)^n (ep_1p_2^2)$, $\mathsf L_B (b) = \{2, n+1\}$, and hence $n-1 \in \Delta (B)$. It can be checked that $\max \Delta (B) \le n$, and then the assertion follows.

\medskip
4. We provide an example of a weakly Krull monoid $B$  with $v$-class group $G$ satisfying all assumptions of  Theorem \ref{1.1} where
\[
\max \Delta (B)  > \mathsf D (G)-2 > \max \Delta (G)
\]
(confer the bounds given in Proposition \ref{2.4}). Since $\mathsf D (G)-2 > \max \Delta (G)$, Proposition \ref{2.4}.2 implies that $G$ can  neither cyclic nor be an elementary $2$-group. We set $G = C_3^r$ with $r \ge 2$,  choose a basis $(e_1, \ldots, e_r)$ of $G$ with $\ord (e_1)= \ldots = \ord (e_r)=3$, and set $e_0=e_1+\ldots+e_r$. Then $\mathsf D (G) = \mathsf D^* (G)=2r+1$. For $i \in [0,r]$, we define a seminormal $v$-noetherian finitely primary monoid
\[
D_i = \{p_i^{k_i}q_i^{l_i} \mid k_i, l_i \in \N \} \cup \{1\} \subset \widehat{ D_i} = \mathcal F ( \{p_i, q_i\}) \,,
\]
and we define a homomorphism $\iota \colon T=D_0 \times \ldots \times D_r \to G$ by $\iota (p_i)=e_i$, $\iota (q_i)=-e_i$ for every $i \in [0,r]$. Then $B=\mathcal B (G,T,\iota)$ is a seminormal $v$-noetherian weakly Krull monoid with nontrivial conductor and $v$-class group isomorphic to $G$ (see Proposition \ref{2.2}). The elements
\[
U_1=p_1q_1^{2} \cdot \ldots \cdot  p_r q_r^{2} \cdot p_0^{2}q_0 \,, \  U_2=p_1^{2}q_1 \cdot \ldots \cdot p_r^{2}q_r \cdot p_0 q_0^{2} \,, \quad \text{and} \quad V_i = p_iq_i \quad \text{for every} \ i \in [0,r]
\]
are atoms of $B$, $U_1U_2= (V_0 \cdot \ldots \cdot V_r)^3$, and $\mathsf L_B (U_1U_2)=\{2,3(r+1) \}$.
\end{examples}


\providecommand{\bysame}{\leavevmode\hbox to3em{\hrulefill}\thinspace}
\providecommand{\MR}{\relax\ifhmode\unskip\space\fi MR }
\providecommand{\MRhref}[2]{%
  \href{http://www.ams.org/mathscinet-getitem?mr=#1}{#2}
}
\providecommand{\href}[2]{#2}


\begin{thebibliography}{10}

\bibitem{An-An-Za92b}
D.D. Anderson, D.F. Anderson, and M.~Zafrullah, \emph{Atomic domains in which
  almost all atoms are prime}, Commun. Algebra \textbf{20} (1992), 1447 --
  1462.

\bibitem{An-Mo-Za92}
D.D. Anderson, J.~Mott, and M.~Zafrullah, \emph{Finite character
  representations for integral domains}, Boll. Unione Mat. Ital. \textbf{6}
  (1992), 613 -- 630.

\bibitem{An-Ch-Pa06a}
D.F. Anderson, Gyu~Whan Chang, and J.~Park, \emph{Weakly {K}rull and related
  domains of the form ${D+M}$, ${A+XB[X]}$ and ${A + X^2B[X]}$}, Rocky Mountain
  J. Math. \textbf{36} (2006), 1 -- 22.

\bibitem{B-C-R-S-S10}
P.~Baginski, S.T. Chapman, R.~Rodriguez, G.J. Schaeffer, and Y.~She, \emph{On
  the delta set and catenary degree of {K}rull monoids with infinite cyclic
  divisor class group}, J. Pure Appl. Algebra \textbf{214} (2010), 1334 --
  1339.

\bibitem{Ba94}
V.~Barucci, \emph{Seminormal {M}ori domains}, Commutative {R}ing {T}heory,
  Lect. Notes Pure Appl. Math., vol. 153, Marcel Dekker, 1994, pp.~1 -- 12.

\bibitem{B-C-K-R06}
C.~Bowles, S.T. Chapman, N.~Kaplan, and D.~Reiser, \emph{On delta sets of
  numerical monoids}, J. Algebra Appl. \textbf{5} (2006), 695 -- 718.

\bibitem{Ch09a}
Gyu~Whan Chang, \emph{Semigroup rings as weakly factorial domains}, Commun.
  Algebra \textbf{37} (2009), 3278 -- 3287.

\bibitem{C-C-M-M-P14}
S.T. Chapman, M.~Corrales, A.~Miller, Ch. Miller, and Dh. Patel, \emph{The
  catenary and tame degrees on a numerical monoid are eventually periodic}, J.
  Australian Math. Soc. \textbf{97} (2014), 289 -- 300.

\bibitem{C-G-L09}
S.T. Chapman, P.A. Garc{\'i}a-S{\'a}nchez, and D.~Llena, \emph{The catenary and
  tame degree of numerical monoids}, Forum Math. \textbf{21} (2009), 117 --
  129.

\bibitem{Ch-Go-Pe14a}
S.T. Chapman, F.~Gotti, and R.~Pelayo, \emph{On delta sets and their realizable
  subsets in {K}rull monoids with cyclic class groups}, Colloq. Math.
  \textbf{137} (2014), 137 -- 146.

\bibitem{Ch-Sc-Sm08b}
S.T. Chapman, W.A. Schmid, and W.W. Smith, \emph{On minimal distances in
  {K}rull monoids with infinite class group}, Bull. Lond. Math. Soc.
  \textbf{40} (2008), 613 -- 618.

\bibitem{Co-Ka15a}
S.~Colton and N.~Kaplan, \emph{The realization problem for delta sets of
  numerical monoids}, \href{http://arxiv.org/abs/1503.08496}{arXiv:1503.08496}.

\bibitem{Co05a}
J.~Coykendall, \emph{Extensions of half-factorial domains{\rm \,:} a survey},
  Arithmetical {P}roperties of {C}ommutative {R}ings and {M}onoids, Lect. Notes
  Pure Appl. Math., vol. 241, Chapman \& Hall/CRC, 2005, pp.~46 -- 70.

\bibitem{Do-Fo87}
D.E. Dobbs and M.~Fontana, \emph{Seminormal rings generated by algebraic
  integers}, Mathematika \textbf{34} (1987), 141 -- 154.

\bibitem{Ge-Gr-Sc-Sc10}
A.~Geroldinger, D.J. Grynkiewicz, G.J. Schaeffer, and W.A. Schmid, \emph{On the
  arithmetic of {K}rull monoids with infinite cyclic class group}, J. Pure
  Appl. Algebra \textbf{214} (2010), 2219 -- 2250.

\bibitem{Ge-Gr-Sc11a}
A.~Geroldinger, D.J. Grynkiewicz, and W.A. Schmid, \emph{The catenary degree of
  {K}rull monoids {I}}, J. Th{\'e}or. Nombres Bordx. \textbf{23} (2011), 137 --
  169.

\bibitem{Ge-HK06a}
A.~Geroldinger and F.~Halter-Koch, \emph{Non-{U}nique {F}actorizations.
  {A}lgebraic, {C}ombinatorial and {A}nalytic {T}heory}, Pure and Applied
  Mathematics, vol. 278, Chapman \& Hall/CRC, 2006.

\bibitem{Ge-Ka-Re15a}
A.~Geroldinger, F.~Kainrath, and A.~Reinhart, \emph{Arithmetic of seminormal
  weakly {K}rull monoids and domains}, J. Algebra \textbf{444} (2015), 201 --
  245.

\bibitem{Ge-Ru09}
A.~Geroldinger and I.~Ruzsa, \emph{Combinatorial {N}umber {T}heory and
  {A}dditive {G}roup {T}heory}, Advanced Courses in Mathematics - CRM
  Barcelona, Birkh{\"a}user, 2009.

\bibitem{Ge-Yu12b}
A.~Geroldinger and P.~Yuan, \emph{The set of distances in {K}rull monoids},
  Bull. Lond. Math. Soc. \textbf{44} (2012), 1203 –-- 1208.

\bibitem{Ge-Zh16a}
A.~Geroldinger and Qinghai Zhong, \emph{The set of minimal distances in {K}rull
  monoids}, Acta Arith., to appear.

\bibitem{Ge-Zh15b}
\bysame, \emph{The catenary degree of {K}rull monoids {II}}, J. Australian
  Math. Soc. \textbf{98} (2015), 324 -- 354.

\bibitem{Gr13a}
D.J. Grynkiewicz, \emph{Structural {A}dditive {T}heory}, Developments in
  Mathematics, Springer, 2013.

\bibitem{HK95a}
F.~Halter-Koch, \emph{Divisor theories with primary elements and weakly {K}rull
  domains}, Boll. Un. Mat. Ital. B \textbf{9} (1995), 417 -- 441.

\bibitem{HK98}
\bysame, \emph{Ideal {S}ystems. {A}n {I}ntroduction to {M}ultiplicative {I}deal
  {T}heory}, Marcel Dekker, 1998.

\bibitem{Ha04c}
W.~Hassler, \emph{Arithmetic of weakly {K}rull domains}, Commun. Algebra
  \textbf{32} (2004), 955 -- 968.

\bibitem{Li12a}
{Jung Wook} {Lim}, \emph{Weakly {K}rull domains and the composite numerical
  semigroup ring ${D}+{E}[{\Gamma^*}]$}, Pacific J. Math. \textbf{257} (2012),
  227 -- 242.

\bibitem{Ma-Ok16a}
P.~Malcolmson and F.~Okoh, \emph{Half-factorial subrings of factorial domains},
  J. Pure Appl. Algebra \textbf{220} (2016), 877 –-- 891.

\bibitem{Ph12b}
A.~Philipp, \emph{A precise result on the arithmetic of non-principal orders in
  algebraic number fields}, J. Algebra Appl. \textbf{11}, 1250087, 42pp.

\bibitem{Pl-Sc16a}
A.~Plagne and W.A. Schmid, \emph{On congruence half-factorial {K}rull monoids
  with cyclic class group}, submitted.

\bibitem{Pl-Sc05a}
\bysame, \emph{On large half-factorial sets in elementary $p$-groups{\rm \,:}
  maximal cardinality and structural characterization}, Isr. J. Math.
  \textbf{145} (2005), 285 -- 310.

\bibitem{Pl-Sc05b}
\bysame, \emph{On the maximal cardinality of half-factorial sets in cyclic
  groups}, Math. Ann. \textbf{333} (2005), 759 -- 785.

\end{thebibliography}
\end{document}